\documentclass[12pt]{amsart}
\usepackage{amssymb,amsmath,amsthm}
\usepackage{tikz}
\usetikzlibrary{arrows}
\usepackage[colorlinks=true,urlcolor=blue,linkcolor=blue,citecolor=blue]{hyperref}

\makeatletter
\pgfarrowsdeclare{stealthp}{stealthp}
{
  \pgfutil@tempdima=0.28pt%
  \advance\pgfutil@tempdima by.3\pgflinewidth%
  \pgfutil@tempdimb=6\pgfutil@tempdima\advance\pgfutil@tempdimb by.5\pgflinewidth%
  \pgfarrowsleftextend{+-\pgfutil@tempdimb}
  \pgfutil@tempdimb=2\pgfutil@tempdima\advance\pgfutil@tempdimb by0.5\pgflinewidth%
  \pgfarrowsrightextend{+\pgfutil@tempdimb}
}
{
  \pgfutil@tempdima=0.56pt%
  \advance\pgfutil@tempdima by.3\pgflinewidth%
  \pgfsetdash{}{+0pt}
  \pgfsetroundjoin
  \pgfpathmoveto{\pgfqpoint{2\pgfutil@tempdima}{0\pgfutil@tempdima}}
  \pgfpathcurveto
  {\pgfqpoint{-.5\pgfutil@tempdima}{.5\pgfutil@tempdima}}
  {\pgfqpoint{-3\pgfutil@tempdima}{1.5\pgfutil@tempdima}}
  {\pgfqpoint{-6\pgfutil@tempdima}{3.25\pgfutil@tempdima}}
  \pgfpathcurveto
  {\pgfqpoint{-3\pgfutil@tempdima}{1\pgfutil@tempdima}}
  {\pgfqpoint{-3\pgfutil@tempdima}{-1\pgfutil@tempdima}}
  {\pgfqpoint{-6\pgfutil@tempdima}{-3.25\pgfutil@tempdima}}
  \pgfpathcurveto
  {\pgfqpoint{-3\pgfutil@tempdima}{-1.5\pgfutil@tempdima}}
  {\pgfqpoint{-.5\pgfutil@tempdima}{-.5\pgfutil@tempdima}}
  {\pgfqpoint{2\pgfutil@tempdima}{0\pgfutil@tempdima}}
  \pgfpathclose
  \pgfusepathqfillstroke
}
\makeatother

\theoremstyle{plain}
\newtheorem{thm}{Theorem}[section]
\newtheorem{lem}[thm]{Lemma}
\newtheorem{conj}[thm]{Conjecture}
\newtheorem{cor}[thm]{Corollary}
\newtheorem{obs}[thm]{Observation}
\newtheorem{prop}[thm]{Proposition}
\newtheorem{question}[thm]{Question}
\theoremstyle{definition}
\newtheorem{defn}[thm]{Definition}

\newcommand{\minprol}{\textnormal{\textsf{minprol}}}

\newcommand{\bm}[1]{\mbox{\boldmath $#1$}}

\newcounter{todocounter}

\newcommand{\DDD}{\mathcal{D}}
\newcommand{\PPP}{\mathcal{P}}
\newcommand{\TTT}{\mathcal{T}}

\newcommand{\bbR}{\mathbb{R}}

\newcommand{\leqs}{\leqslant}
\newcommand{\geqs}{\geqslant}
\newcommand{\veps}{\varepsilon}
\newcommand{\ceil}[1]{\left\lceil #1 \right\rceil}
\newcommand{\floor}[1]{\left\lfloor #1 \right\rfloor}
\newcommand{\varempty}{\scalebox{1.2}{$\varnothing$}}

\newcommand{\perm}{\sigma}
\newcommand{\permb}{\tau}
\newcommand{\pat}{\pi}
\newcommand{\dw}{d_\perm}
\newcommand{\dwk}{d_{\perm_k}}
\newcommand{\dwp}{d_{\wdel{p}}}
\newcommand{\br}{\textnormal{\textsf{\textup{br}}}}
\newcommand{\wdel}[1]{\perm_{\langle #1 \rangle}}
\newcommand{\nred}{\textnormal{\textsf{\textup{nred}}}}
\newcommand{\nblue}{\textnormal{\textsf{\textup{nblue}}}}
\newcommand{\discrep}{\delta}
\newcommand{\gridlattice}{\Gamma}

\tikzset{help lines/.style={color=gray!67,very thin}}

\newcommand{\plotptradius}{0.275}
\newcommand{\setplotptradius}[1]{\renewcommand{\plotptradius}{#1}}

\newcommand{\plotpt}[3][] 
{ \fill[#1,radius=\plotptradius] (#2,#3) circle; }

\newcommand{\plotpermnobox}[3][]  
{
  \foreach \y [count=\x] in {#3}
  {
    \ifnum0=\y {} \else {
      \plotpt[#1]{\x}{\y}
    } \fi
  }
}

\title[Prolific permutations and permuted packings]{Prolific permutations and permuted packings: \\
downsets containing many large patterns}

\author[D. Bevan]{David Bevan}
\address{Department of Computer and Information Sciences, University of Strathclyde, Glasgow G1~1XH Scotland}
\email{david.bevan@strath.ac.uk}

\author[C. Homberger]{Cheyne Homberger}
\address{Department of Mathematics and Statistics, University of Maryland, Baltimore County, Baltimore, MD 21250}
\email{cheyneh@umbc.edu}

\author[B. E. Tenner]{Bridget Eileen Tenner$^{\dagger}$}
\address{Department of Mathematical Sciences, DePaul University, Chicago, IL 60614}
\email{bridget@math.depaul.edu}
\thanks{$^{\dagger}$Research partially supported by a Simons Foundation Collaboration Grant for Mathematicians and by a DePaul University Faculty Summer Research Grant.}


\begin{document}

\begin{abstract}

A permutation of $n$ letters is $k$-prolific if each $(n-k)$-subset of the letters in its one-line notation forms a unique pattern.  We present a complete characterization of $k$-prolific permutations for each $k$, proving that $k$-prolific permutations of $m$ letters exist for every $m \geqslant k^2/2+2k+1$, and that none exist of smaller size.  Key to these results is a natural bijection between $k$-prolific permutations and certain ``permuted'' packings of diamonds.

\end{abstract}

\maketitle

\section{Introduction}

The set of permutations of $[n] = \{1,2,\ldots, n\}$ is denoted $S_n$. We write a permutation $\perm\in S_n$ as a word over $[n]$ in one-line notation, $\perm = \perm(1)\perm(2)\cdots \perm(n)$, and say that such a permutation $\perm$ has \emph{size} $n$.
If $\pat_1$ and $\pat_2$ are words of the same size over $\mathbb{R}$, then we write $\pat_1\approx\pat_2$ to denote that their letters appear in the same relative order.
This prompts the classical notion of pattern containment.

\begin{defn}
Consider $\pat \in S_r$. A permutation $\perm \in S_n$ \emph{contains} the \emph{pattern} $\pat$ if there are indices $1 \leqs i_1 < \cdots < i_r \leqs n$ such that $\perm(i_1)\cdots \perm(i_r) \approx \pat$.
If $\perm$ contains $\pat$, we write $\pat \preceq \perm$.
If $\perm$ does not contain $\pat$, then $\perm$ \emph{avoids} $\pat$.
\end{defn}

From this, it is natural to define the ``pattern poset'' on permutations.
\begin{defn}
Let the \emph{pattern poset}, $\PPP$,  be the graded poset
over $\bigcup_{k\geqs1}S_k$,
ordered by the containment relation $\preceq$.
\end{defn}
By definition, the elements of rank $k$ in $\PPP$ are exactly the elements of $S_k$.

This paper is concerned with principal downsets of this poset, that is, with the sets of patterns which lie below a given permutation. In particular, we examine those permutations whose downset is as large as possible in the upper ranks.

This is related to problems of \emph{pattern packing}~\cite{albert:packing,miller:packing}, which seek to maximize the total number of distinct patterns contained in a permutation, and to problems of
\emph{superpatterns}~\cite{eriksson:superpattern,Gunby2014,Hegarty2013,miller:packing}, which are concerned with determining the size of the smallest permutations whose downset contains every permutation of some fixed size.
Other related work addresses permutation \emph{reconstruction}~\cite{Ginsburg2007,Raykova2006,smith:reconstruction}, establishing when permutations are uniquely determined by the (multi)set of large patterns they contain.
The reader is referred to the books by B\'ona~\cite{Bona2012} and Kitaev~\cite{Kitaev2011} for an overview of problems related to the permutation pattern poset.

It follows immediately from the definition of $\PPP$ that, for a permutation $\perm \in S_n$, there are at most $\binom{n}{k}$ distinct permutations $\pi \preceq \perm$ that lie exactly $k$ ranks below $\perm$ in $\PPP$, since each such permutation is obtained from $\perm$ by the deletion of $k$ letters from the one-line notation for $\perm$.
Our interest is in those permutations of size $n$ which contain maximally many patterns of size $n-k$.
\begin{defn}
Fix positive integers $n > k \geqs1$. A permutation $\perm \in S_n$ is \emph{$k$-prolific} if
$$
\Big| \big\{\pat \in S_{n-k} \::\: \pat \preceq \perm \big\} \Big| \;=\; \binom{n}{k}.
$$
\end{defn}
Clearly, not every permutation is $k$-prolific. As a trivial example, the identity permutation $12\cdots n \in S_n$ contains only one pattern of each size, and thus is never $k$-prolific for any $k<n$.

Prolific permutations were previously investigated by the second author in \cite{homberger:plentiful}. The present work corrects and significantly improves upon the results presented there.

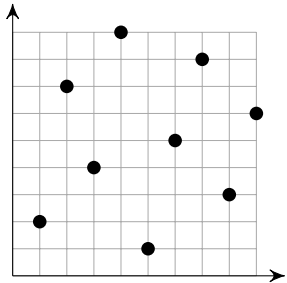
\begin{figure}[htbp]
  $$
  \begin{tikzpicture}[scale=0.36,>=stealthp] 
    \draw[help lines] (0,0) grid (9,9);
    \setplotptradius{0.25}
    \plotpermnobox[black]{}{2,7,4,9,1,5,8,3,6}
    \draw[->] (0,0) -- (10,0);
    \draw[->] (0,0) -- (0,10);
  \end{tikzpicture}
  $$
  \caption{The plot of the permutation $274915836$.}
  \label{fig:plot}
\end{figure}

It is helpful to consider permutations from a graphical perspective.

\begin{defn}\label{defn:plot}
The \emph{plot} of a permutation $\perm\in S_n$ is the collection of lattice points $\{(i, \perm(i)):1\leqs i\leqs n\}$ in the Euclidean plane $\bbR^2$. In practice, we tend to identify the $i$th entry of a permutation $\perm$ with the \emph{point} $(i, \perm(i))$ in its plot, and we  linearly order the points in a plot from left to right; that is,
$$(i,\perm(i)) \;<\; (j,\perm(j)),$$
if $i < j$.
\end{defn}

See Figure~\ref{fig:plot} for an illustration of a permutation plot.

This viewpoint motivates the following two definitions concerning the distance between entries of a permutation.

\begin{defn}
For a permutation $\perm \in S_n$ and $i,j \in [n]$,
the \emph{distance} $\dw(i,j)$ between the $i$th and $j$th entries of $\perm$ is given by the $L_1$
distance (the ``taxicab'' or ``Manhattan'' distance) between the corresponding points in the plot of $\perm$:
$$
\dw(i,j) \;=\;
\big|\!\big| (i,\perm(i)) - (j,\perm(j)) \big|\!\big|_1
\;=\;
\big|i - j\big| \:+\: \big|\perm(i) - \perm(j)\big|.
$$
\end{defn}
For example, if $\perm=274915836$, as in Figure~\ref{fig:plot}, then $\dw(1,2)=1+5=6$, and $\dw(1,3)=2+2=4$.

\begin{defn}
For $n>1$, the \emph{breadth} of $\perm \in S_n$, denoted $\br(\perm)$, is the minimum  distance between any two distinct entries:
$$
\br(\perm)
\;=\;
\min_{i,j\in[n],\,i\neq j} \dw(i,j).
$$
\end{defn}
For example, $\br(274915836)=4$, and this is realized by any of the pairs of entries $$\{i,j\} \;\in\; \{\{1,3\},\{2,3\},\{2,4\},\{3,6\},\{4,7\},\{6,7\},\{6,8\},\{6,9\},\{7,9\},\{8,9\}\}.$$

With these definitions in place, we can state our two primary results. First, we have the following complete characterization of $k$-prolific permutations (Theorem~\ref{thm:kprolific iff large breadth}):
\begin{center}
  A permutation~$\perm$ is $k$-prolific if and only if $\br(\perm)\geqs k+2$.
\end{center}
That is, permutations are prolific precisely if their points are not too close together.
(Coleman~\cite{Coleman2004} was the first to observe that maximising the distance between points tends to increase the number of distinct subpermutations.)
As a consequence, it is readily seen that $k$-prolific permutations of size $n$ are in bijection with certain packings of diamonds,
which we call \emph{permuted packings}.
Section~\ref{section:characterization} is dedicated to the proof of this theorem. (This result was previously presented in~\cite{homberger:plentiful}, but the short proof given there contains an error.)

It is not possible for small permutations to be $k$-prolific because their points are too close together.
Hence, our second main result is an exact determination of the minimum possible size of a $k$-prolific permutation (Corollary~\ref{cor:minprol}):
\begin{center}
The smallest $k$-prolific permutations have size $\ceil{k^2/2+2k+1}$.
\end{center}
In Section~\ref{section:lower bound},
we prove that every $k$-prolific permutation
must be at least this big (Theorem~\ref{thm:minprol lower bound}).
The argument
relies heavily on the interpretation of $k$-prolific permutations as permuted diamond packings.
Then, in Section~\ref{section:constructions}, we present constructions demonstrating that $k$-prolific permutations do exist of this size (Theorem~\ref{thm:sigmak is kprolific}), and also of all greater sizes (Theorem~\ref{thm:kprolific from minprol}).

In the final section of the paper we discuss possible directions for further research, including some questions concerning permuted packings which may be of independent interest.

\section{\texorpdfstring{Characterizing $k$-prolific permutations}{Characterizing k-prolific permutations}}
\label{section:characterization}

We begin by introducing notation to denote the pattern that results from the deletion of specified entries from a permutation.

\begin{defn}
For a permutation $\perm \in S_n$ and $i\in[n]$, let
$\wdel{i} \in S_{n-1}$
be the pattern formed by deleting the $i$th entry from $\perm$.
Similarly, if $A = \{i_1,i_2,\ldots, i_k\}\subset[n]$, then let $\wdel{A} \in S_{n-k}$ be the pattern formed by deleting the $i_1$th, $i_2$th, $\ldots$, $i_k$th entries from $\perm$.
\end{defn}

The goal of this section is to prove that a permutation is $k$-prolific if and only if its breadth is at least $k+2$.
Specifically, we need to demonstrate that, given a permutation $\perm\in S_n$, there exist distinct $k$-sets of indices $A, B \subset [n]$ such that $\wdel{A}=\wdel{B}$ if and only if $\br(\perm)<k+2$.

The proof of the ``only if'' direction is straightforward (and was first proved by Hegarty~\cite{Hegarty2013}), as is the argument in the ``if'' direction when there is an index common to $A$ and $B$;
we present these later.
The next several pages, leading up to Lemma~\ref{lem:disjoint A and B} are thus concerned with characterizing the situation when $\wdel{A}=\wdel{B}$
with $A$ and $B$ disjoint.
To this end, we introduce a plane graph associated with such a scenario and determine its structure.

To define this graph, we first
need to define what it means for an entry in a permutation to ``fulfill'' an entry in a pattern that it contains.

\begin{defn}
Suppose $\perm \in S_n$ and $A\subset[n]$.
Let $[n]\setminus A = \{i_1,i_2, \ldots,i_r\}$, where $i_1 < i_2 < \cdots < i_r$.
For each $j\in[r]$, we say that the $i_j$th entry of $\perm$ \emph{fulfills} the $j$th entry of $\wdel{A}$.
\end{defn}

Our graph joins the points of $\perm$ that fulfill the ``same'' point in $\wdel{A}$ and $\wdel{B}$.

\begin{defn}\label{defn:chain graph}
Given a permutation $\perm\in S_n$, and disjoint $k$-sets of indices $A, B \subset [n]$,
such that $\wdel{A}=\wdel{B}$,
the \emph{chain graph} of $\perm$ for $A$ and $B$ is a plane graph on the points in the plot of $\perm$.
For each index $i \in [n - k]$, an edge is added between the point of $\perm$ that fulfills the $i$th entry of $\wdel{A}$ and the point of $\perm$ that fulfills the $i$th entry of $\wdel{B}$.
If $\wdel{A}(i)$ and $\wdel{B}(i)$ are fulfilled by the same point, $p$, of $\perm$, then we call $p$ a \emph{fixed point}, and no edge is added.

To facilitate the discussion, we let the vertices corresponding to elements of $A$ be coloured \emph{red}, and those corresponding to elements of $B$ be coloured \emph{blue}. The remaining vertices are \emph{uncoloured}.
\end{defn}

Note that this definition implies that no vertex of a chain graph has degree greater than two.
See Figure~\ref{fig:chain graph} for an illustration of a chain graph; its vertex set contains eight red points (in~$A$), eight blue points (in~$B$), six fixed points, and seventeen other uncoloured points.

\begin{figure}[htbp]
  $$
  \begin{tikzpicture}[scale=0.24,line join=round,>=stealthp] 
    \draw[help lines] (1,1) grid (39,39);
    \draw [thick,->] (3,4)--(5,7); \draw [thick,->] (5,7)--(8,12); \draw [thick,->] (8,12)--(12,17); \draw [thick,->] (12,17)--(16,21); \draw [thick,->] (16,21)--(20,25); \draw [thick,->] (20,25)--(23,29);
    \draw [thick,->] (4,11)--(6,16); \draw [thick,->] (6,16)--(9,20); \draw [thick,->] (9,20)--(13,24); \draw [thick,->] (13,24)--(17,28);
    \draw [thick,->] (7,5)--(11,8); \draw [thick,->] (11,8)--(15,13);
    \draw [thick,->] (19,18)--(22,22); \draw [thick,->] (22,22)--(24,26);
    \draw [thick,->] (10,3)--(14,6); \draw [thick,->] (14,6)--(18,10); \draw [thick,->] (18,10)--(21,15);
    \draw [thick,->] (29,38)--(27,36); \draw [thick,->] (27,36)--(26,34);
    \draw [thick,->] (30,37)--(28,35);
    \draw [thick,->] (39,9)--(38,14); \draw [thick,->] (38,14)--(37,19); \draw [thick,->] (37,19)--(36,23); \draw [thick,->] (36,23)--(35,27); \draw [thick,->] (35,27)--(34,30);
    \setplotptradius{0.4}
    \plotpermnobox[black]{}{32, 39,  4, 11,  7, 16,  5, 12, 20,  3,  8, 17, 24,  6, 13, 21, 28, 10, 18, 25, 15, 22, 29, 26, 31, 34, 36, 35, 38, 37,  2, 33,  1, 30, 27, 23, 19, 14,  9}
    \setplotptradius{0.5}
    \plotpermnobox[red]  {}{ 0,  0,  4, 11,  0,  0,  5,  0,  0,  3,  0,  0,  0,  0,  0,  0,  0,  0, 18,  0,  0,  0,  0,  0,  0,  0,  0,  0, 38, 37,  0,  0,  0,  0,  0,  0,  0,  0,  9}
    \plotpermnobox[blue] {}{ 0,  0,  0,  0,  0,  0,  0,  0,  0,  0,  0,  0,  0,  0, 13,  0, 28,  0,  0,  0, 15,  0, 29, 26,  0, 34,  0, 35,  0,  0,  0,  0,  0, 30,  0,  0,  0,  0,  0}
    \setplotptradius{0.25}
    \plotpermnobox[white] {}{ 0,  0,  0,  0,  0,  0,  0,  0,  0,  0,  0,  0,  0,  0, 13,  0, 28,  0,  0,  0, 15,  0, 29, 26,  0, 34,  0, 35,  0,  0,  0,  0,  0, 30,  0,  0,  0,  0,  0}

    \draw [gray] (1,-5)--(40,-5);
    \draw [very thick] (1,-5)--(3,-5)--(5,-3)--(7,-3)--(8,-2)--(10,-2)--(11,-1)--(15,-1)--(16,-2)--(17,-2)--(18,-3)--(19,-3)--(20,-2)--(21,-2)--(22,-3)--(23,-3)--(25,-5)--(26,-5)--(27,-6)--(28,-6)--(29,-7)--(31,-5)--(34,-5)--(35,-6)--(39,-6)--(40,-5);

    \setplotptradius{0.25}
    \plotpermnobox[black]{}{-5, -5,   0,  0, -3, -3,  0, -2, -2,  0, -1, -1, -1, -1,  0, -2,  0, -3,  0, -2,  0, -3,  0,  0, -5,  0, -6,  0,  0,  0, -5, -5, -5,  0, -6, -6, -6, -6,  0}
    \setplotptradius{0.35}
    \plotpermnobox[red]  {}{ 0,  0,  -5, -4,  0,  0, -3,  0,  0, -2,  0,  0,  0,  0,  0,  0,  0,  0, -3,  0,  0,  0,  0,  0,  0,  0,  0,  0, -7, -6,  0,  0,  0,  0,  0, 0,   0,  0, -6}
    \plotpermnobox[blue] {}{ 0,  0,   0,  0,  0,  0,  0,  0,  0,  0,  0,  0,  0,  0, -1,  0, -2,  0,  0,  0, -2,  0, -3, -4,  0, -5,  0, -6,  0,  0,  0,  0,  0, -5,  0,  0,   0,  0,  0}
    \setplotptradius{0.15}
    \plotpermnobox[white] {}{ 0,  0,   0,  0,  0,  0,  0,  0,  0,  0,  0,  0,  0,  0, -1,  0, -2,  0,  0,  0, -2,  0, -3, -4,  0, -5,  0, -6,  0,  0,  0,  0,  0, -5,  0,  0,   0,  0,  0}
    \end{tikzpicture}
  $$
  \caption{An oriented chain graph (Definition~\ref{defn:chain graph}), and a plot of the discrepancy (Definition~\ref{defn:discrepancy}) of its vertices.
  The edges of each chain are oriented away from its red end-vertex, shown as a disk, towards its blue end-vertex, shown as a ring.}
  \label{fig:chain graph}
\end{figure}
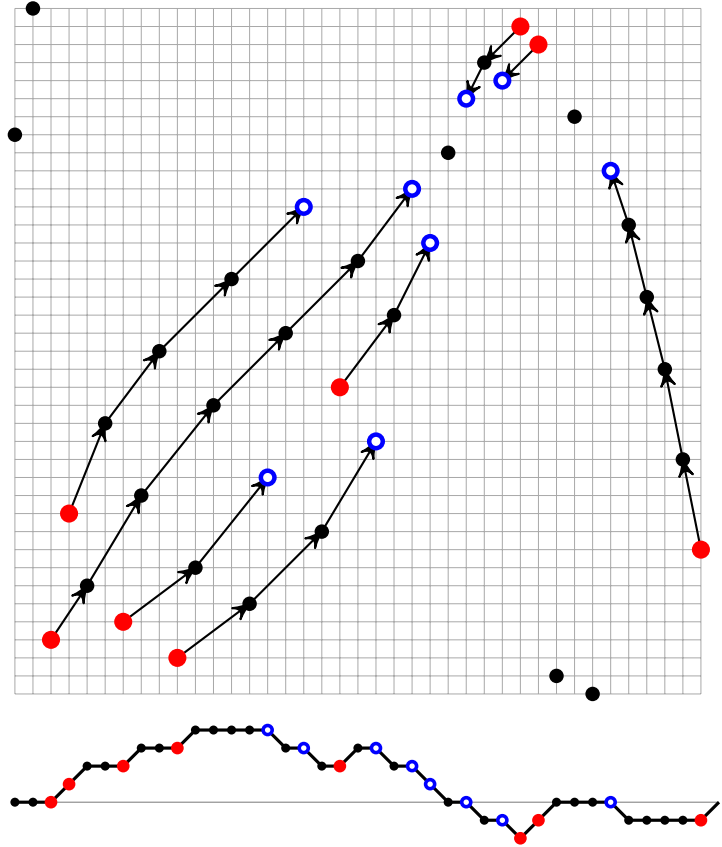

Recall the comment in Definition~\ref{defn:plot} about identifying points in the plot of a permutation with their $x$-coordinates, and ordering them from left to right.
In a chain graph, these points are the vertices. Thus, the vertices of a chain graph are
also identified by their $x$-coordinates, and considered to be ordered from left to right.

Note that
throughout this section we only consider situations in which $A\cap B=\varempty$ and  $\wdel{A}=\wdel{B}=\pat$, say.
The definition of a chain graph is restricted to these settings.
In such cases, if
the $a$th entry of~$\perm$ fulfills the $i$th entry of $\wdel{A}$ and the $b$th entry of~$\perm$ fulfills the $i$th entry of $\wdel{B}$,
then it is also the case that
the $\perm(a)$th entry of~$\perm^{-1}$ fulfills the $\pat(i)$th entry of ${\wdel{A}}^{\!\!-1}$ and the $\perm(b)$th entry of~$\perm^{-1}$ fulfills the $\pat(i)$th entry of ${\wdel{B}}^{\!\!-1}$.
Hence, properties of the chain graph are preserved under permutation inversion,
so symmetry may be invoked to convert ``horizontal'' arguments into ``vertical'' ones.

We are interested in determining the properties of chain graphs, with the ultimate goal of proving that their vertices cannot be too far apart.
To investigate their structure, we introduce the concept of the ``discrepancy'' of a vertex in a chain graph.

\begin{defn}\label{defn:discrepancy}
For each point $p$ in a chain graph, let $\nred(p)$ be the number of red points strictly to the left of~$p$ and $\nblue(p)$ be the number of blue points strictly to the left of~$p$. We define the \emph{discrepancy} of $p$, which we denote $\discrep(p)$, to be the difference, $\discrep(p) = \nred(p) - \nblue(p)$.
\end{defn}

A plot showing the discrepancy of the vertices in a chain graph is exhibited at the bottom of Figure~\ref{fig:chain graph}.

Discrepancy has the properties outlined in the following two observations.

\begin{obs}
If we consider the points from left to right, the discrepancy either increases by one (after a red point), decreases by one (after a blue point), or stays the same (after an uncoloured point).
Since there are equally many red points as blue points, the discrepancy returns to zero after the last point.
\end{obs}
Thus a plot of the discrepancy is a \emph{Motzkin bridge} --- similar to a Motzkin path,
but permitted to wander both above and below its start point.

\begin{obs}\label{obs:fulfillment discrepancy}
If the point $p$ (the $p$th entry of $\perm$) is not red, then it fulfills the entry of $\wdel{A}$ that has index $a:=p-\nred(p)$.
Similarly, if $p$ is not blue, then it fulfills the entry of $\wdel{B}$ that has index $b:=p-\nblue(p)$.
Hence, an uncoloured point $p$ is a fixed point if and only if $\discrep(p)=0$.

If $p$ is not red and $\discrep(p) > 0$,
then
the $a$th entry of $\wdel{B}$ is fulfilled by the $\discrep(p)$th non-blue point to the left of $p$.
Similarly, if $p$ is not blue and $\discrep(p) \geqs 0$, then
the $b$th entry of $\wdel{A}$ is fulfilled by the $\discrep(p+1)th$ non-red point to the right of $p$.
The case for negative $\discrep(p)$ is analogous; indeed, it is equivalent to the positive case applied to the reverse of the permutation.
\end{obs}

We now show that the structure of a chain graph is tightly constrained, and that its name is justified.

\begin{prop}\label{prop:chain graph structure}
The chain graph of a permutation $\perm$ for $k$-sets $A$ and $B$ consists of
$k$ monotone paths,
which we call
\emph{chains} 
together with one isolated vertex for each fixed point.
Each chain has one red end-vertex and one blue.

Suppose $\ell$ and $r$ are the left and right end-vertices, respectively, of a chain $C$.
If $\ell$ is red, then $\discrep(\ell)\geqs0$ and for every point $q$ of $\perm$ such that $\ell<q\leqs r$, we have $\discrep(q)>0$.
Analogously, if $\ell$ is blue, then $\discrep(\ell)\leqs0$ and for every point $q$ of $\perm$ such that $\ell<q\leqs r$, we have $\discrep(q)<0$.
\end{prop}
\begin{proof}
Firstly, each fixed point has degree 0, by definition.

Secondly, each red vertex has degree 1 since it fulfills a point of $\wdel{B}$, but does not fulfill a point of $\wdel{A}$. Analogously, each blue vertex also has degree~1.

Thirdly, let $p$ be a non-fixed uncoloured vertex.
By Observation~\ref{obs:fulfillment discrepancy}, we know that
$\discrep(p)\neq0$ and that
$p$ is adjacent to two other vertices, one to its left (which we denote $p^-$), and the other to its right (denoted~$p^+$).
So $p$ is a medial vertex in a path whose vertices are ordered from left to right.
By symmetry, $p$ is also medial in a path whose vertices are ordered from bottom to top. Thus, from left to right, the path is either monotonically increasing or monotonically decreasing.

Suppose now that $p$ is not red and $\discrep(p)>0$.
If there are $n_r$ red, $n_b$ blue and $n_u$ uncoloured points in $[p^-,p)$, then
$\discrep(p^-) = \discrep(p) - n_r + n_b$.
Since $p^-$ is the $\discrep(p)$th non-blue point to the left of $p$, we also have $\discrep(p)=n_r+n_u$, so $n_r\leqs\discrep(p)$.
Thus, $\discrep(p^-)=n_b+n_u\geqs0$.
Moreover, $\discrep(p^-)=0$ if and only if the $\discrep(p)$ points immediately to the left of $p$, including $p^-$, are all red.
Furthermore,
since there are only $n_r\leqs\discrep(p)$ up-steps in the plot of the discrepancy between $p^-$ and $p$,
for all $q\in(p^-,p)$, we have $\discrep(q)>0$.

Provided $p^-$ is uncoloured and $\discrep(p^-)>0$, we may repeat this argument until either $p^-$ is coloured or
$\discrep(p^-)=0$. In either case, since $p^-$ is, by definition, not blue, we see that $p^-$ is red. Thus
the left end-vertex, $\ell$, of the chain containing $p$, is red, and $\discrep(q)>0$ for all
$q\in(\ell,p]$.

An analogous argument shows that if $p$ is not fixed and not blue and $\discrep(p)\geqs0$ then
$\discrep(p^+)>0$, and for all $q\in(p,p^+)$, we have $\discrep(q)>0$.
By iterating this, we see that the right end-vertex, $r$, of the chain containing $p$, is blue, and $\discrep(q)>0$ for all
$q\in(p,r]$.

As before, the argument for negative $\discrep(p)$ is equivalent to the positive case applied to the reverse of the permutation.
\end{proof}

As we go forward, in referring to edges in a chain, we make use of the following notation, as used in the proof of Proposition~\ref{prop:chain graph structure}.

\begin{defn}
Given a point $p$ in a chain, but not at its rightmost end, write $p^+$ for the point adjacent to $p$ on its right.
Thus, every edge of a chain is $pp^+$ for some point $p$.
\end{defn}

In the light of Proposition~\ref{prop:chain graph structure}, we distinguish between \emph{increasing} and \emph{decreasing} chains as follows.

\begin{defn}
  A path in a chain graph is an \emph{increasing chain} if each of its edges $pp^+$ satisfies $\perm(p)<\perm(p^+)$.
  A path in a chain graph is a \emph{decreasing chain} if each of its edges $pp^+$ satisfies $\perm(p)>\perm(p^+)$.
\end{defn}

As an example, the chain graph illustrated in Figure~\ref{fig:chain graph} consists of seven increasing chains, one decreasing chain and six isolated fixed points.

Before stating further properties of chain graphs, we introduce the ``span'' and ``central span'' of a pair of points, and the idea of a point ``cutting'' an edge.

\begin{defn}
Fix $\perm \in S_n$. For any distinct $i,j \in [n]$, the \emph{span} of $i$ and $j$ in $\perm$ is the set of entries of $\perm$ whose positions lie strictly between $i$ and $j$ \emph{or} whose values lie strictly between $\perm(i)$ and $\perm(j)$. The \emph{central span} of $i$ and $j$ in $\perm$ is the set of entries of $\perm$ whose positions lie strictly between $i$ and~$j$ \emph{and} whose values lie strictly between $\perm(i)$ and $\perm(j)$.
\end{defn}

Thus, the span of two points consists of the points in a cross-shaped region, and the central span of two points consists of the points in a rectangular region.

\begin{defn}
Let $q$ be in the span of points $p$ and $p'$. This $q$ \emph{cuts}
$pp'$ \emph{from the left} (respectively, \emph{right}) if $q$'s position is to the left (respectively, right) of both $p$ and $p'$. This $q$ \emph{cuts} $pp'$ \emph{from below} (respectively, \emph{above}) if $q$'s value is less (respectively, greater) than both $\perm(p)$ and $\perm(p')$. Cuts from the left or right are \emph{horizontal}; cuts from below or above are \emph{vertical}.
Points in the central span of $p$ and $p'$ are considered to cut $pp'$ both horizontally and vertically.
\end{defn}

There is a close relationship between cutting and the distance between two points.

\begin{obs}
The distance $\dw(p,p')$ is two greater than the number of times that $pp'$ is cut.
\end{obs}

With these definitions in place, we can state two elementary corollaries of Proposition~\ref{prop:chain graph structure}. These are
the first of several results characterizing how edges of chains may be cut, which we use later to prove that some pair of
points in a chain graph must be close together.
\begin{cor}\label{cor:cut same chain}
An edge of a chain $C$ cannot be cut by another point from $C$.
\end{cor}
\begin{cor}\label{cor:cut fixed point}
An
edge of a chain cannot be cut by a fixed point.
\end{cor}
\begin{proof}
Every point strictly between the end-vertices of a chain has non-zero discrepancy, whereas a fixed point has zero discrepancy. Thus no fixed point can cut a chain vertically.
The horizontal case follows by symmetry.
\end{proof}

In light of the fact that one end-vertex of each chain is red and the other is blue, it makes sense to orient the edges of a chain graph.
We choose to orient the edges of each chain away from its red end-vertex and towards its blue end-vertex.

\begin{defn}
We say that a chain is \emph{oriented leftwards}, \emph{rightwards}, \emph{upwards} or \emph{downwards} according to whether its blue end-vertex is to the left of, to the right of, above or below its red end-vertex, respectively.
\end{defn}
In Figure~\ref{fig:chain graph}, the orientation of a chain is shown using arrows;
three chains (two increasing and one decreasing) are oriented leftwards and five (all increasing) are oriented rightwards, while six chains (five increasing and one decreasing) are oriented upwards and two (both increasing) downwards.

We now show that chains in an oriented chain graph are further constrained by having to satisfy a ``consistent orientation'' property.
\begin{defn}
Suppose that $C$ and $C'$ are chains in the chain graph of a permutation $\perm$, with
left end-vertices $\ell$ and $\ell'$ and right end-vertices $r$ and $r'$, respectively.
We say that $C$ and $C'$ \emph{overlap horizontally} if $\ell < r'$ and $\ell' < r$.
Analogously, 
suppose that $C$ and $C'$ have 
lower end-vertices $b$ and $b'$ and upper end-vertices $t$ and $t'$, respectively.
We say that $C$ and $C'$ \emph{overlap vertically} if $\perm(b) < \perm(t')$ and $\perm(b') < \perm(t)$.
\end{defn}

\begin{prop}
If two chains in a chain graph overlap horizontally, then either both chains are oriented leftwards or both chains are oriented rightwards.
Analogously, if two chains overlap vertically, then either they are both oriented upwards or they are both oriented downwards.
\end{prop}
\begin{proof}
Let $C$ and $C'$ be chains in the chain graph of a permutation $\perm$, with
left end-vertices $\ell$ and $\ell'$ and right end-vertices $r$ and $r'$, respectively.
Without loss of generality, suppose that $\ell$ is red. By Proposition~\ref{prop:chain graph structure}, for every point $q$ of~$\perm$ that lies strictly between $\ell$ and $r$, we have $\discrep(q) > 0$.
Now, if $\ell'$ were blue, then, similarly, for every point $q'$ of $C'$, we would have $\discrep(q') \leqs 0$.
But, since $C$ and $C'$ overlap horizontally, there is some point of $C'$ between the end-vertices of $C$,
so $\ell'$ must in fact be red.

The vertical case follows by symmetry.
\end{proof}

This consistent orientation property has consequences for how edges of chains may be cut.
\begin{cor}
If a chain $C$ has an edge cut vertically by a point from another chain $C_1$, then either $C$ and $C_1$ are both oriented leftwards or they are both oriented
rightwards. Analogously,
if $C$ has an edge cut horizontally by a point from a chain $C_2$, then either $C$ and $C_2$ are both oriented upwards or they are both oriented
downwards.
\end{cor}

\begin{cor}\label{cor:cut incr decr chains}
If $C$ is an increasing chain and $C'$ is a decreasing chain, then it is not possible for $C$ and~$C'$ to overlap both horizontally and vertically.
Thus, an edge of an increasing chain~$C$ cannot be cut both horizontally and vertically by points from a decreasing chain $C'$.
\end{cor}

\begin{cor}\label{cor:cut double incr and decr chain}
If $C_1$ and $C_2$ are increasing chains that overlap (either horizontally or vertically), and $C'$ is a decreasing chain,
then it is not possible for there to be points $q_1$ and $q_2$ of $C'$ such that $q_1$ cuts an edge of $C_1$ horizontally and $q_2$ cuts an edge of $C_2$ vertically.
\end{cor}

Chains in a chain graph also satisfy an ``interleaving'' property, which implies that two chains cannot cross. Recall the linear ordering of points from Definition~\ref{defn:plot}.

\begin{prop}
Suppose $pp^+$ and $qq^+$ are edges in distinct chains.
If $p<q$, then $p^+<q^+$.
If $\perm(p)<\perm(q)$, then $\perm(p^+)<\perm(q^+)$.
\end{prop}

\begin{proof}
If $p^+<q$, then the result follows trivially. Assume that $q<p^+$. Without loss of generality, suppose that $\discrep(p^+)>0$, and hence $\discrep(q^+)>0$.

Suppose $p^+$ were to the right of $q^+$. Let $d$ be the difference between their $x$-coordinates, and let $n_b$ be the number of blue points in the interval $[q^+,p^+)$.
Then, $\discrep(p^+)\leqs\discrep(q^+)+d-2n_b$, with a strict inequality if $n_b=0$, since $q^+$ cannot be red.

Now, point $p$ is the $\discrep(p^+)$th non-blue point to the left of $p^+$.
Since there are $d-n_b$ non-blue points in $[q^+,p^+)$, it is the case that $p$ is
no further to the left than the ($\discrep(q^+)-n_b)$th non-blue point to the left of~$q^+$.
But $q$ is the $\discrep(q^+)$th non-blue point to the left of $q^+$, which means that $p$ is to the right of~$q$,
a contradiction.

The vertical case follows by symmetry.
\end{proof}

The interleaving property further restricts the ways in which an edge of a chain may be cut by points from another chain.

\begin{cor}\label{cor:cut horiz vert}
An edge $e$ of a chain $C$ can be cut at most once horizontally and at most once vertically by points from some other chain~$C'$.
If $C$ and $C'$ are increasing chains, then it is only possible for points of $C'$ to cut $e$ either from the left and from above (if $C'$ is to the upper left of $C$), or else from the right and from below (if $C'$ is to the lower right of $C$).
\end{cor}

Corollary~\ref{cor:cut horiz vert}, together with Corollaries~\ref{cor:cut same chain}, \ref{cor:cut fixed point}, \ref{cor:cut incr decr chains}, and~\ref{cor:cut double incr and decr chain}, completes our characterization of how edges of chains may be cut.
We are now able to prove that the points in a chain graph cannot be very far apart.

\begin{lem}\label{lem:disjoint A and B}
Suppose we have a permutation $\perm$, and disjoint $k$-sets of indices $A$ and $B$,
such that $\wdel{A}=\wdel{B}$.
Then $\br(\perm) < k+2$.
\end{lem}
\begin{proof}
Assume, to the contrary, that $\br(\perm) \geqs k+2$ and let $G$ be the chain graph of $\perm$ for $A$ and~$B$.

To begin, we demonstrate that every edge of a chain in $G$ must be cut twice, once horizontally and once vertically, by points from some other chain.
Let
$e=pp^+$ be an edge of a chain $C$ of $G$.
Without loss of generality, suppose that $C$ is an increasing chain.

Since $\dw(p, p^+) \geqs k + 2$, the edge $e$ is cut by at least $k$ vertices.
By Corollaries~\ref{cor:cut same chain} and~\ref{cor:cut fixed point}, $e$~is not cut by a point of $C$ or by a fixed point.
There are only $k-1$ chains in $G$ that are distinct from~$C$, so, by the pigeonhole principle, there is at least one chain, $C'$, whose points cut $e$ twice.
By Corollary~\ref{cor:cut incr decr chains}, $C'$ is an increasing chain, and
by Corollary~\ref{cor:cut horiz vert}, one cut must be horizontal and the other vertical.

Now, we show that the two cuts cannot be from the same point of $C'$.
Suppose that a point $q$ of $C'$ cuts $e$ both horizontally and vertically; that is, this $q$ is in the central span of $p$ and $p^+$.
Now, $e$ can be cut by at most $k - 1$ points horizontally, one from each chain, and similarly by at most $k - 1$ points vertically, so
$\dw(p,p^+)\leqs 2k$.
Thus, since
$\dw(p,p^+)=\dw(p, q)+\dw(q, p^+)$,
either
$\dw(p, q) \leqs k$ or $\dw(q, p^+) \leqs k$, and so $\br(\perm)\leqs k$, a contradiction.
Thus, there is no point in the central span of $p$ and $p^+$, and
$e$ is cut by two distinct points of~$C'$.

To conclude our proof, we examine the positioning of the points that occur in increasing chains, and by an infinite descent argument reach a contradiction.
Specifically, we prove that if $\br(\perm) \geqs k+2$ then to the upper left of each increasing chain is another increasing chain, which is an impossibility since the number of chains is finite.
For brevity in what follows, we call a point of an increasing chain an \emph{increasing point}, and a point of a decreasing chain a \emph{decreasing point}.

Suppose that $G$ contains $m$ increasing chains and $d:=k-m$ decreasing chains.
Without loss of generality, we may assume that $m\geqs1$. For an illustration of the argument that follows,
see Figure~\ref{fig:lemma disjoint A and B} (for simplicity, $d=0$ in this figure).

\begin{figure}[ht]
  $$
  \begin{tikzpicture}[scale=0.3,line join=round]
    \draw[help lines] (1,1) grid (23,19);
    \draw[very thick] (13,1)--(17,3); \draw[thick] (17,3)--(20,6)--(22,10)--(23,14);
    \node [below] at (13,.7) {$p_1$};
    \node [below right] at (17,3.7) {$p_1^+$};
    \node [] at (14.55,2.9) {$e_1$};
    \draw[very thick] (8,2)--(11,5); \draw[thick] (11,5)--(14,9)--(18,13)--(21,17);
    \node [below] at (8,1.7) {$p_2$};
    \node [below right] at (11,5.7) {$p_2^+$};
    \node [] at (8.8,4.2) {$e_2$};
    \draw[very thick] (4,4)--(6,8); \draw[thick] (6,8)--(9,12)--(12,16)--(15,19);
    \node [below] at (4,3.7) {$p_3$};
    \node [below right] at (6,8.7) {$p_3^+$};
    \node [] at (4.1,6.45) {$e_3$};
    \draw[red,ultra thick] (1,7)--(2,11); \draw[thick] (2,11)--(3,15)--(5,18);
    \node [below] at (1,6.7) {$p_4$};
    \node [below right] at (2,11.7) {$p_4^+$};
    \node [] at (0.5,9.25) {$e_4$};
    \plotpermnobox[black]{}{7,11,15,4,18,8,0,2,12,0,5,16,1,9,19,0,3,13,0,6,17,10,14}
  \end{tikzpicture}
  $$
  \caption{An illustration of the second half of the proof of Lemma~\ref{lem:disjoint A and B}, for $k=m=4$.}
  \label{fig:lemma disjoint A and B}
\end{figure}
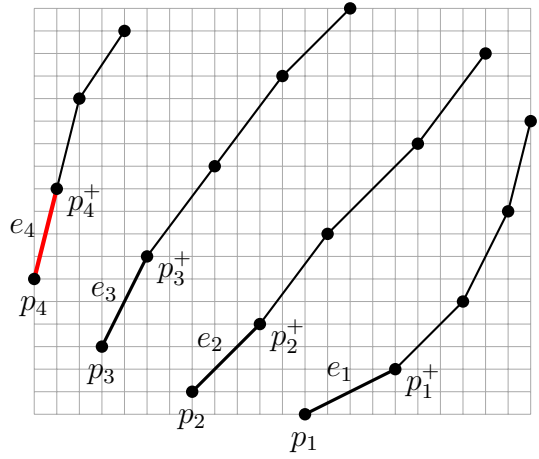

Let $p_1$ be the lowermost increasing point in $G$, and let $e_1$ be the edge $p_1p_1^+$, with $p_1^+$ to the upper right of $p_1$.
Then, for each $j=2,\ldots,m$, let $p_j$ be the lowermost increasing point that cuts $e_{j-1}$ from the left (we show that such a point always exists),
and set $e_j := p_jp_j^+$. Our goal is to prove that, for each $j$, the edge $e_j$
is not cut from below by any increasing point.
We proceed by induction on~$j$.

For the base case, it follows from the definition of $p_1$
and the fact that the central span of $p_1$ and $p_1^+$ is empty
that $e_1$ is not cut from
below by any increasing point.

Now, fix $j\geqs1$ and assume that $e_j$ is
not cut from below by any increasing point.

Since $e_j$ is not cut from below by an increasing point,
no increasing chain has points which cut $e_j$ from both the right and from below. So, by Corollary~\ref{cor:cut horiz vert},
points of some increasing chain must cut $e_j$ from both the left and from above.
So there exists a lowermost increasing point that cuts $e_j$ from the left, namely $p_{j+1}$.

Suppose that
$e_j$ is cut from the right by $r_j$ increasing points
and cut horizontally by $h_j$ decreasing points.
By the definition of $p_{j+1}$, these are the only points that can occur above $p_j$ and below $p_{j+1}$, so
$\perm(p_{j+1}) - \perm(p_j) \leqs r_j + h_j + 1$. Hence, since $\dw(p_j,p_{j+1})\geqs k+2$, we have
$$
p_j - p_{j+1} \;\geqs\; (k+2) \:-\: (r_j + h_j + 1) \;=\; k+1-r_j-h_j.
$$
Now suppose that $e_{j+1}$ is cut vertically by $v_{j+1}$ decreasing points. By Proposition~\ref{cor:cut double incr and decr chain}, we know that $h_j+v_{j+1}\leqs d$.
Observe that $e_{j+1}$ is cut from above by at most $m-1-r_j$ increasing points, the $r_j$ chains cutting $e_j$ from the right also being to the right of $e_{j+1}$.
Hence,
$$
p_{j+1}^+-p_{j+1} \;\leqs\; m-1-r_j+v_{j+1} + 1 \;\leqs\; m-r_j+d-h_j \;=\; k-r_j-h_j.
$$
Thus,
$$
p_j-p_{j+1}^+ \;\geqs\; (k+1-r_j-h_j) \:-\: (k-r_j-h_j) \;=\; 1.
$$
Thus, $p_{j+1}^+$ is to the left of $p_j$, so $e_{j+1}$ is not cut from below by $p_j$.

Moreover,
since the central span of $p_{j+1}$ and $p_{j+1}^+$ is empty, the edge
$e_{j+1}$ is not cut from below by any other increasing point. This is because, otherwise,
either $p_1$ would not be the lowermost increasing point in $G$, or else
$p_2,\ldots,p_{j+1}$ would not be the lowermost increasing points cutting $e_1,\ldots,e_j$ from the left.

Now consider the final edge $e_m=p_mp_m^+$. By the same inductive argument, it too is cut from the left by some increasing point $p_{m+1}$. But this is impossible, because there are only $m$ increasing chains in $G$.
Hence, our initial assumption is false, and thus $\br(\perm) < k+2$.
\end{proof}

Having established the desired result when $A$ and $B$ are disjoint,
we now have almost all we need to establish the relationship between the breadth of a permutation and whether that permutation is $k$-prolific or not.
The final ingredient is the following proposition, adapted from~\cite{homberger:plentiful}.

\begin{prop}\label{prop:breadth after deletion}
Deleting a single entry from a permutation decreases the breadth by at most one.
\end{prop}

\begin{proof}
Suppose $\perm$ is a permutation with breadth $b$, and $p$ is a point of $\perm$.
Now, for any $i$ and~$j$, we have $\dwp(i,j)\geqs\dw(i,j)-2$, where $\dwp(i,j)=\dw(i,j)-2$ precisely when $p$ is in the central span of $i$ and $j$ in $\perm$.
Hence, $\br(\wdel{p})=b-2$ only if there are points $i$ and $j$ such that $\dw(i,j) = b$ and $p$ is in their central span. But $p$ cannot be in their central span, for then we would have $\dw(i,p)<\dw(i,j)=b$, and the breadth of $\perm$ would be less than~$b$.
\end{proof}

Here, finally, is our first main result: a complete characterization of $k$-prolific permutations.

\begin{thm}\label{thm:kprolific iff large breadth}
A permutation $\perm$ is $k$-prolific if and only if $\br(\perm)\geqs k+2$.
\end{thm}

\begin{proof}
For the forward direction, suppose that $i \neq j$ are such that $\dw(i,j) < k + 2$.
Let $S$ be the set of elements of~$\perm$ in the span of $i$ and $j$.
It follows that $|S| < k$. But we have $\wdel{S \cup \{i\}}=\wdel{S \cup \{j\}}$, so $\perm$ is not $(|S|+1)$-prolific.
Moreover, $\wdel{S \cup \{i\}\cup X}=\wdel{S \cup \{j\}\cup X}$, for any set of indices $X$ not containing $i$ or $j$, so $\sigma$ is not $k$-prolific.

For the reverse direction, we proceed by induction on $k$.

For the base case, suppose that $\br(\perm)\geqs3$, but that $\perm$ is not $1$-prolific; that is, there exists $i \neq j$ such that $\wdel{i}  = \wdel{j}$.
Assume without loss of generality that $i < j$ and $\perm(i) < \perm(j)$.
The $(j-1)$th entry of $\wdel{i}$ is $\perm(j) - 1$, while the $(j-1)$th entry of $\wdel{j} $ is $\perm(j-1)$ or $\perm(j-1)-1$, depending on the relationship between $\perm(j-1)$ and $\perm(j)$.
The latter case can be disregarded since it would imply that $\perm(j-1)=\perm(j)$, an impossibility.

But if $\perm(j)-1 = \perm(j-1)$, then
$$
\br(\perm) \;\leqs\; \dw(j,j-1) \;=\; \big|j - (j-1)\big| \:+\: \big|\perm(j) - \perm(j-1)\big| \;=\; 2,
$$ a contradiction. Therefore $\perm$ must be $1$-prolific.

Now fix $k > 1$ and assume that, for any permutation $\permb$, if $\br(\permb)\geqs k+1$, then $\permb$ is $(k-1)$-prolific.
Suppose that the breadth of $\perm$ is at least $k+2$, but that $\perm$ is not $k$-prolific; that is, there are distinct $k$-sets $A$ and $B$ such that $\wdel{A} = \wdel{B}$.

If there is an index $c \in A \cap B$, then $\perm' = \wdel{c}$ is $(k-1)$-prolific, by Proposition~\ref{prop:breadth after deletion} and the induction hypothesis.
But
$$
\perm'_{\langle A \setminus \{c\}\rangle} \;=\; \wdel{A} \;=\; \wdel{B} \;=\; \perm'_{\langle B \setminus \{c\}\rangle},
$$
so $\perm'$ cannot be $(k-1)$-prolific.
Thus $A$ and $B$ must be disjoint.

The result then follows by Lemma~\ref{lem:disjoint A and B}.
\end{proof}

As a consequence of this characterization, we see that any permutation of size $n$
containing maximally many patterns of size $n-k$ also contains maximally many larger patterns.

\begin{cor}\label{cor:prolific is hereditary}
  If $\perm$ is $k$-prolific, then $\perm$ is also $j$-prolific for all $1\leqs j< k$.
\end{cor}

\section{\texorpdfstring{Bounding the size of $k$-prolific permutations from below}{Bounding the size of k-prolific permutations from below}}
\label{section:lower bound}

In this section, we
determine a lower bound on the size of $k$-prolific permutations.
We use the following notation to denote the size of the smallest $k$-prolific permutation.

\begin{defn}
Given a positive integer $k$, let $\minprol(k)$ be the minimum value $n$ for which there exists a $k$-prolific permutation in $S_n$.
\end{defn}

Clearly,
for a $k$-prolific permutation of size $n$ to exist, we need
$(n-k)!\geqs\binom{n}{k}$. This inequality yields a very weak lower bound on $\minprol(k)$, which can, using Stirling's approximation, be shown to grow like $k+e\sqrt{k}$ for large $k$.
Nevertheless, in conjunction with
the fact that both 2413 and 3142 cover all four non-monotone permutations in $S_3$, it is sufficient to determine that
$\minprol(1)=4$.

We now establish a much tighter lower bound on $\minprol(k)$ by recasting
Theorem~\ref{thm:kprolific iff large breadth} in terms of packings of diamonds.
Recall that a \emph{translational packing}
with a tile (a compact non-empty subset of $\bbR^2$) is a collection of translates of the tile
whose interiors are mutually disjoint (see~\cite{brass:discretegeometry}).
We are interested in translational packings in which the tiles are centered on the points of a permutation.

\begin{defn}
  A translational packing $\Pi$, consisting of $n$ translates of a tile $\TTT$, is
  a \emph{permuted packing} if there exists a permutation $\perm\in S_n$ such that
  $\Pi = \{ \TTT+(i,\perm(i)) : 1\leqs i\leqs n \}$.
\end{defn}

The following proposition establishes the relationship between $k$-prolific permutations and permuted packings.

\begin{prop}
Given integers $n>k\geqs1$,
let $\DDD$ be a diamond whose diagonal has length $k+2$.
The family of $k$-prolific permutations of size $n$ is equinumerous to
the family of permuted packings that consist of $n$ translates of $\DDD$.
\end{prop}

\begin{proof}
  By Theorem~\ref{thm:kprolific iff large breadth},
  a permutation is $k$-prolific if and only if the minimum $L_1$ distance between two
  points in the plot of the permutation is at least $k+2$.
  Thus, if we center a ball of radius $k/2+1$ (under the $L_1$ metric) at each point of the plot, the interiors of these balls are mutually disjoint.
  Since, in $\bbR^2$, an $L_1$ ball of radius $k/2+1$ is a diamond whose diagonal has length $k+2$, it is readily seen that this construction yields a bijection between $k$-prolific permutations of size $n$ and the specified family of permuted packings.
\end{proof}

Figure~\ref{fig:diamond packing even} depicts a 6-prolific permutation and the corresponding permuted diamond packing.

\begin{figure}[htbp]
\begin{tikzpicture}[scale=.24]
\foreach \x in
{(1,6),(2,19),(3,28),(4,13),(5,2),(6,23),(7,32),(8,17),(9,10),(10,27),(11,4),(12,21),(13,14),(14,31),(15,8),(16,25),(17,18),(18,3),(19,12),(20,29),(21,22),(22,7),(23,16),(24,1),(25,26),(26,11),(27,20),(28,5),(29,30),(30,15),(31,24),(32,9)}
{
    \fill[blue!18] \x ++(0,4) -- ++(4,-4) -- ++(-4,-4) -- +(-4,4);
}
\begin{scope}
  \clip (3,3) rectangle (30,30);
\foreach \x in
{(1,6),(2,19),(3,28),(4,13),(5,2),(6,23),(7,32),(8,17),(9,10),(10,27),(11,4),(12,21),(13,14),(14,31),(15,8),(16,25),(17,18),(18,3),(19,12),(20,29),(21,22),(22,7),(23,16),(24,1),(25,26),(26,11),(27,20),(28,5),(29,30),(30,15),(31,24),(32,9)}
{
    \fill[blue!9] \x ++(0,4) -- ++(4,-4) -- ++(-4,-4) -- +(-4,4);
}\end{scope}
\foreach \x in
{(1,6),(2,19),(3,28),(4,13),(5,2),(6,23),(7,32),(8,17),(9,10),(10,27),(11,4),(12,21),(13,14),(14,31),(15,8),(16,25),(17,18),(18,3),(19,12),(20,29),(21,22),(22,7),(23,16),(24,1),(25,26),(26,11),(27,20),(28,5),(29,30),(30,15),(31,24),(32,9)}
{
    \fill \x circle (0.275);
    \draw[thick,gray!50!black] \x ++(0,4) -- ++(4,-4) -- ++(-4,-4) -- ++(-4,4) -- ++(4,4);
}
\draw (-3,3)--(36,3);
\draw (-3,30)--(36,30);
\draw (3,-3)--(3,36);
\draw (30,-3)--(30,36);
\draw[<->] (8.4,17) -- (11.6,17);
\draw (10,17.1) node[below] {\scriptsize 4};
\draw[<->] (8,17.4) -- (8,20.6);
\draw (8.2,19) node[left] {\scriptsize 4};
\draw (1,6) node[left] {\scriptsize $a$};
\draw (2,19) node[left] {\scriptsize $b$};
\draw (2.7,13) node[left] {\scriptsize $b$};
\draw (3.2,2) node[left] {\scriptsize $a$};
\end{tikzpicture}
\caption{The permuted diamond packing corresponding to a $6$-prolific permutation, showing the square box used in the proof of Theorem~\ref{thm:minprol lower bound}.}
\label{fig:diamond packing even}
\end{figure}
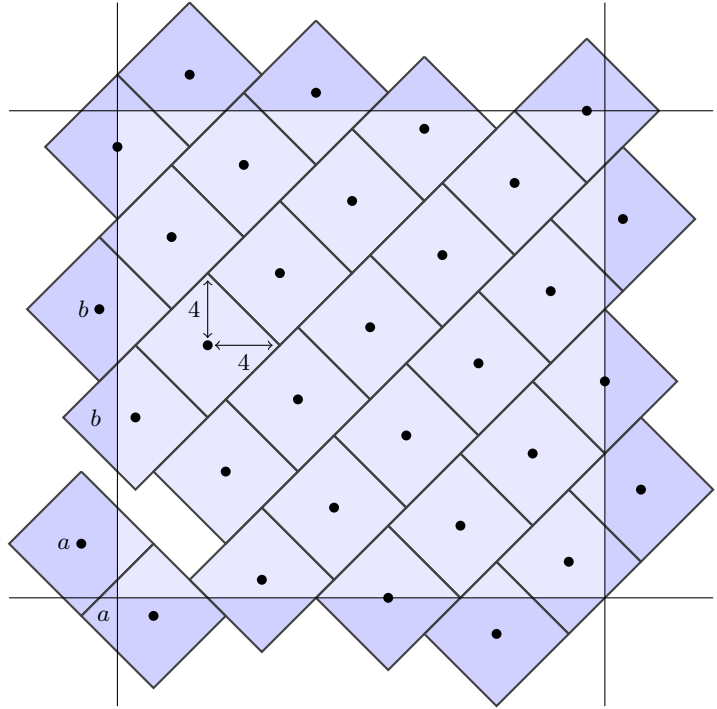

Using this characterization of $k$-prolific permutations, we now bound $\minprol(k)$ from below.

\begin{thm}\label{thm:minprol lower bound}
For each positive integer $k$,
$
\minprol(k) \;\geqs\; \ceil{k^2/2+2k+1}.
$
\end{thm}

This result was previously proved for $k > 800$ by Gunby~\cite{Gunby2014}.
The basic approach taken in our proof was first used by Miller~\cite{miller:packing}.

\begin{proof}
Suppose that $\perm\in S_n$ is $k$-prolific.
Let $s=k/2+1$ be the length of the semidiagonals of the diamonds in the associated permuted packing.
The area of each of the $n$ diamonds in the packing is $2s^2$.

Consider the square box $[s-1,n+2-s]^2$, centered over the packing, as illustrated in Figure~\ref{fig:diamond packing even}.
The margins around this box have width $k$.
Note that when $k$ is even, $s$ is an integer and the sides of the box pass through the centers of four of the diamonds;
when $k$ is odd, they do not.

The total area of the diamond tiles is bounded above by the area of this square box plus the total area of the parts of the diamonds that ``overflow'' into the margins outside the box. The overflowing parts are shaded more darkly in Figure~\ref{fig:diamond packing even}.
The area of the parts of the diamonds in a given margin can be calculated exactly, as follows.

Consider the region to the left of the left side of the box, including the top left and bottom left corners. In this margin are parts of each of the diamonds centered at the first $k$ points of $\perm$.
For each $j < k/2$, the overflowing part of the $j$th diamond from the left and the overflowing part of the $(k-j)$th diamond can be glued together to form a complete diamond. For example, in Figure~\ref{fig:diamond packing even}, the two parts labelled $a$ form a complete diamond, as do the two labelled $b$.
If $k$ is even, then the overflowing part of the $(k/2)$th diamond is exactly half a diamond. Finally, the overflowing part of the $k$th diamond is a triangle with area $1$.
Thus, the total area of the parts of the diamonds that overflow into any one margin is given by
$$
\left(\tfrac{k-1}{2}\right)\cdot 2s^2 + 1 \;=\; (k-1)s^2 + 1.
$$

The total area of all of the parts of the diamonds that overflow is no more than four times this, a value which counts the contributions from the corners twice. Therefore, we have the inequality
$$
2s^2n \;\leqs\; (n+3-2s)^2 \:+\: 4\big((k-1)s^2+1\big).
$$
Substituting $k/2+1$ for $s$ and applying the quadratic formula, we obtain
$$
n \;\geqs\; \frac{k^2 + 8k + \sqrt{k^4 + 32k - 16}}{4},
$$
an expression which exceeds $k^2/2 + 2k$ for all $k > 1/2$.

Since $\minprol(k)$ is an integer, for even $k$ we thus have $\minprol(k) \geqs k^2/2 + 2k + 1$, as required, and
for odd $k$, $\minprol(k) \geqs k^2/2 + 2k + 1/2$.

\begin{figure}[htbp]
\begin{tikzpicture}[scale=.24]
\foreach \x in
{(7,12),(8,18),(9,3),(10,23),(15,5),(16,11),(17,24),(19,1),(24,4),(25,17),(27,22)}
{
    \fill[blue!18] \x ++(0,3.5) -- ++(3.5,-3.5)  -- ++(.5,-.5)  -- ++(-.5,-.5)  -- ++(-.5,.5) -- ++(-3,-3) -- ++(-3.5,3.5) -- ++(3.5,3.5);
}
\foreach \x in
{(1,9),(2,2),(3,15),(4,21),(5,6),(6,26),(11,8),(12,14),(13,27),(14,19),(18,16),(20,7),(21,20),(22,13),(23,25),(26,10)}
{
    \fill[blue!18] \x ++(0,3.5) -- ++(3,-3)  -- ++(.5,.5)  -- ++(.5,-.5)  -- ++(-.5,-.5) -- ++(-3.5,-3.5) -- ++(-3.5,3.5) -- ++(3.5,3.5);
}
\begin{scope}
  \clip (2.5,2.5) rectangle (25.5,25.5);
\foreach \x in
{(7,12),(8,18),(9,3),(10,23),(15,5),(16,11),(17,24),(19,1),(24,4),(25,17),(27,22)}
{
    \fill[blue!9] \x ++(0,3.5) -- ++(3.5,-3.5)  -- ++(.5,-.5)  -- ++(-.5,-.5)  -- ++(-.5,.5) -- ++(-3,-3) -- ++(-3.5,3.5) -- ++(3.5,3.5);
}
\foreach \x in
{(1,9),(2,2),(3,15),(4,21),(5,6),(6,26),(11,8),(12,14),(13,27),(14,19),(18,16),(20,7),(21,20),(22,13),(23,25),(26,10)}
{
    \fill[blue!9] \x ++(0,3.5) -- ++(3,-3)  -- ++(.5,.5)  -- ++(.5,-.5)  -- ++(-.5,-.5) -- ++(-3.5,-3.5) -- ++(-3.5,3.5) -- ++(3.5,3.5);
}
\end{scope}
\foreach \x in
{(7,12),(8,18),(9,3),(10,23),(15,5),(16,11),(17,24),(19,1),(24,4),(25,17),(27,22)}
{
    \fill \x circle (0.275);
    \draw[thick,gray!50!black] \x ++(0,3.5) -- ++(3.5,-3.5)  -- ++(.5,-.5)  -- ++(-.5,-.5)  -- ++(-.5,.5) -- ++(-3,-3) -- ++(-3.5,3.5) -- ++(3.5,3.5);
}
\foreach \x in
{(1,9),(2,2),(3,15),(4,21),(5,6),(6,26),(11,8),(12,14),(13,27),(14,19),(18,16),(20,7),(21,20),(22,13),(23,25),(26,10)}
{
    \fill \x circle (0.275);
    \draw[thick,gray!50!black] \x ++(0,3.5) -- ++(3,-3)  -- ++(.5,.5)  -- ++(.5,-.5)  -- ++(-.5,-.5) -- ++(-3.5,-3.5) -- ++(-3.5,3.5) -- ++(3.5,3.5);
}
\draw (-2.5,2.5)--(30.5,2.5);
\draw (-2.5,25.5)--(30.5,25.5);
\draw (2.5,-2.5)--(2.5,30.5);
\draw (25.5,-2.5)--(25.5,30.5);
\end{tikzpicture}
\caption{A permuted packing of extended diamonds corresponding to a $5$-prolific permutation.}
\label{fig:diamond packing odd}
\end{figure}
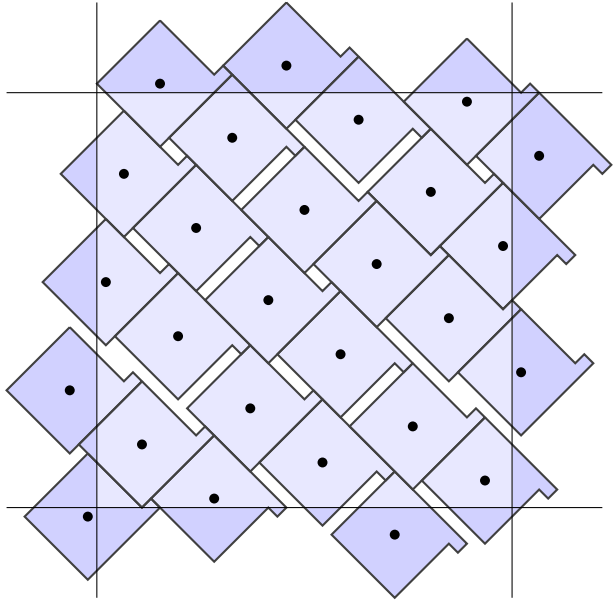
A marginally greater lower bound can be established for odd $k$ by using a slightly different shape of tile.
If $k$ is odd, then the length of the semidiagonal of the diamonds,
$s = k/2 + 1$, is a half-integer.
Since, in the packing, each diamond is centered on an integer lattice point, it is not possible for three of these diamonds to meet at a point.
Thus, either above or below the
rightmost corner of each of these diamonds is a small diamond-shaped region, of semidiagonal length $1/2$, not covered by any diamond tile.
We thus extend the tiles by the addition of these regions, which we call \emph{extensions}, and consider permuted packings of these extended diamonds.
See Figure~\ref{fig:diamond packing odd} for an illustration.

For these extended diamonds, we now repeat our analysis of the
parts of the tiles that overflow into the margins.
The total area of the parts of the diamond-shaped extensions to the right of the box is $k/2+1/4$, made up of $k$ complete extensions, each of area $1/2$, and half an extension from the $(k+1)$th diamond from the right.
The contribution from the extensions that overflow into the top margin is at most $k/4$, made up of $(k-1)/2$ complete extensions and half an extension.
The bottom margin is analogous.
Since no extension can overflow into the left margin, the total area of the parts of the extensions that overflow is no more than $k+1/4$.
Therefore, accounting for the additional area of each tile, we have the inequality
$$
(2s^2+1/2)n \;\leqs\; (n+3-2s)^2 \:+\: 4\big((k-1)s^2+1\big) \:+\: k+1/4.
$$
After substitution for $s$ and the application of the quadratic formula, this yields
$$
n \;\geqs\; \frac{k^2 + 8k + 1 + \sqrt{k^4 + 2k^2 + 32k - 19}}{4},
$$
an expression which exceeds $k^2/2 + 2k + 1/2$ for all $k > 5/8$.

Since $\minprol(k)$ is an integer, for odd $k$ we thus have $\minprol(k) \geqs k^2/2 + 2k + 3/2$, as required.
\end{proof}

\section{\texorpdfstring{Constructions of $k$-prolific permutations}{Constructions of k-prolific permutations}}
\label{section:constructions}

In this section, we
establish that there
is a $k$-prolific permutation of every size greater than or equal to the lower bound of Theorem~\ref{thm:minprol lower bound}, by construction.

Let $m(k)=\ceil{k^2/2+2k+1}$ be the lower bound function from Theorem~\ref{thm:minprol lower bound}.
Our initial constructions enable us to prove that $\minprol(k)=m(k)$.

\begin{defn}\label{defn:a short prolific permutation}
  For each $k\geqs1$, define $\perm_k$ as follows.
  For $i=1,\ldots,m(k)$, let
  $$
  \perm_k(i) \;=\;
  \begin{cases}
  i(k+2) \!\!\mod\, m(k)+1, & \text{ if $k$ is odd, and}\\[2pt]
  i(k+1) \!\!\mod\, m(k)+1, & \text{ if $k$ is even.}
  \end{cases}
  $$
\end{defn}

  See Figure~\ref{fig:sigma5 and sigma6} for an illustration of $\perm_5$ and $\perm_6$.
  We claim that $\perm_k$ is a $k$-prolific permutation of size $m(k)$.

First, we must prove that
$\perm_k$ is indeed a permutation; that is, that $\perm_k(i)$ takes a distinct value for each $i$. To do so, it is sufficient to show that $k+2$ is coprime to $m(k)+1$ when $k$ is odd, and that $k+1$ is coprime to $m(k)+1$ when $k$ is even. Observe that for odd $k$, we have
$$2 \big( k^2/2 + 2k + 5/2 \big) \:+\: (-k-2)\big(k+2\big) \;=\; 1,$$
and for even $k$,
$$2 \big( k^2/2 + 2k + 2\big) \:+\: (-k-3) \big(k+1\big) \;=\; 1.$$
Thus the relevant terms are coprime.

To demonstrate that each $\perm_k$ is $k$-prolific, we use an alternative characterization of the permutations, in terms of two interlocking grids of lattice points as illustrated in Figure~\ref{fig:sigma5 and sigma6}.

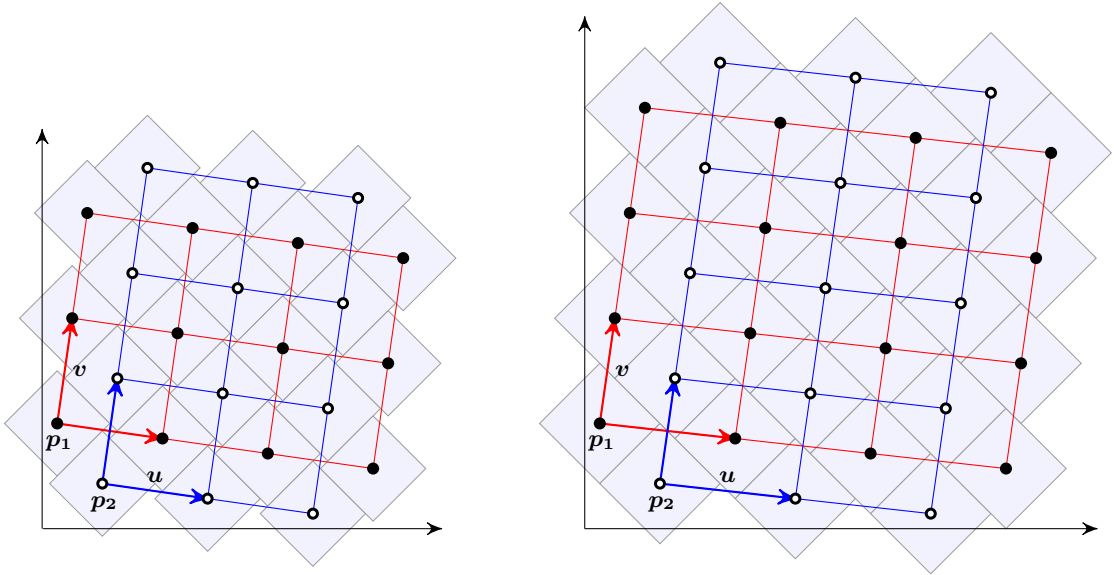
\begin{figure}[htbp]
\begin{tikzpicture}[scale=.2,>=stealthp] 
\foreach \x in
{(1,7),(2,14),(3,21),(4,3),(5,10),(6,17),(7,24),(8,6),(9,13),(10,20),(11,2),(12,9),(13,16),(14,23),(15,5),(16,12),(17,19),(18,1),(19,8),(20,15),(21,22),(22,4),(23,11),(24,18)}
{
    \fill[blue!5] \x ++(0,3.5) -- ++(3.5,-3.5) -- ++(-3.5,-3.5) -- +(-3.5,3.5);
    \draw[gray!75!white] \x ++(0,3.5) -- ++(3.5,-3.5) -- ++(-3.5,-3.5) -- ++(-3.5,3.5) -- ++(3.5,3.5);
}
\draw (1,6.9) node[below] {\scriptsize $\,\bm{p_1}$};
\draw[red] (1,7) -- (22,4);
\draw[red] (2,14) -- (23,11);
\draw[red] (3,21) -- (24,18);
\draw[red] (1,7) -- (3,21);
\draw[red] (8,6) -- (10,20);
\draw[red] (15,5) -- (17,19);
\draw[red] (22,4) -- (24,18);
\draw[thick,red,->] (1,7) -- (8,6);
\draw[thick,red,->] (1,7) -- (2,14);
\draw (1.3,10.5) node[right] {\scriptsize $\bm{v}$};
\draw (4,2.9) node[below] {\scriptsize $\,\bm{p_2}$};
\draw[blue] (4,3) -- (18,1);
\draw[blue] (5,10) -- (19,8);
\draw[blue] (6,17) -- (20,15);
\draw[blue] (7,24) -- (21,22);
\draw[blue] (4,3) -- (7,24);
\draw[blue] (11,2) -- (14,23);
\draw[blue] (18,1) -- (21,22);
\draw[thick,blue,->] (4,3) -- (11,2);
\draw (7.5,2.4) node[above] {\scriptsize $\bm{u}$};
\draw[thick,blue,->] (4,3) -- (5,10);
\foreach \x in
{(1,7),(2,14),(3,21),(4,3),(5,10),(6,17),(7,24),(8,6),(9,13),(10,20),(11,2),(12,9),(13,16),(14,23),(15,5),(16,12),(17,19),(18,1),(19,8),(20,15),(21,22),(22,4),(23,11),(24,18)}
{
    \fill \x circle (0.4);
}
    \draw[->] (0,0) -- (26.5,0);
    \draw[->] (0,0) -- (0,26.5);
    \draw[white] (0,0) -- (3,-3);  
\foreach \x in
{(4,3),(5,10),(6,17),(7,24),(11,2),(12,9),(13,16),(14,23),(18,1),(19,8),(20,15),(21,22)}
{
    \fill [white] \x circle (0.2);
}\end{tikzpicture}
$\qquad$
\begin{tikzpicture}[scale=.2,>=stealthp] 
\foreach \x in
{(1,7),(2,14),(3,21),(4,28),(5,3),(6,10),(7,17),(8,24),(9,31),(10,6),(11,13),(12,20),(13,27),(14,2),(15,9),(16,16),(17,23),(18,30),(19,5),(20,12),(21,19),(22,26),(23,1),(24,8),(25,15),(26,22),(27,29),(28,4),(29,11),(30,18),(31,25)}
{
    \fill[blue!5] \x ++(0,4) -- ++(4,-4) -- ++(-4,-4) -- +(-4,4);
    \draw[gray!75!white] \x ++(0,4) -- ++(4,-4) -- ++(-4,-4) -- ++(-4,4) -- ++(4,4);
}
\draw (1,6.9) node[below] {\scriptsize $\,\bm{p_1}$};
\draw[red] (1,7) -- (28,4);
\draw[red] (2,14) -- (29,11);
\draw[red] (3,21) -- (30,18);
\draw[red] (4,28) -- (31,25);
\draw[red] (1,7) -- (4,28);
\draw[red] (10,6) -- (13,27);
\draw[red] (19,5) -- (22,26);
\draw[red] (28,4) -- (31,25);
\draw[thick,red,->] (1,7) -- (10,6);
\draw[thick,red,->] (1,7) -- (2,14);
\draw (1.3,10.5) node[right] {\scriptsize $\bm{v}$};
\draw (5,2.9) node[below] {\scriptsize $\,\bm{p_2}$};
\draw[blue] (5,3) -- (23,1);
\draw[blue] (6,10) -- (24,8);
\draw[blue] (7,17) -- (25,15);
\draw[blue] (8,24) -- (26,22);
\draw[blue] (9,31) -- (27,29);
\draw[blue] (5,3) -- (9,31);
\draw[blue] (14,2) -- (18,30);
\draw[blue] (23,1) -- (27,29);
\draw[thick,blue,->] (5,3) -- (14,2);
\draw (9.5,2.4) node[above] {\scriptsize $\bm{u}$};
\draw[thick,blue,->] (5,3) -- (6,10);
\foreach \x in
{(1,7),(2,14),(3,21),(4,28),(5,3),(6,10),(7,17),(8,24),(9,31),(10,6),(11,13),(12,20),(13,27),(14,2),(15,9),(16,16),(17,23),(18,30),(19,5),(20,12),(21,19),(22,26),(23,1),(24,8),(25,15),(26,22),(27,29),(28,4),(29,11),(30,18),(31,25)}
{
    \fill \x circle (0.4);
}
    \draw[->] (0,0) -- (34,0);
    \draw[->] (0,0) -- (0,34);
\foreach \x in {(5,3),(6,10),(7,17),(8,24),(9,31),(14,2),(15,9),(16,16),(17,23),(18,30),(23,1),(24,8),(25,15),(26,22),(27,29)}
{
    \fill [white] \x circle (0.2);
}
\end{tikzpicture}
\caption{Plots of the permutations $\perm_5 \in S_{24}$ and $\perm_6 \in S_{31}$, showing their construction from two interlocking grids of lattice points.}
\label{fig:sigma5 and sigma6}
\end{figure}

For odd $k$, define the four vectors
$$
\bm{p_1} \:=\: (1,k+2)
, \quad
\bm{p_2} \:=\: \left(\tfrac{k+3}{2},\tfrac{k+1}{2}\right)
, \quad
\bm{u} \:=\: (k+2,-1)
, \quad
\bm{v} \:=\: (1,k+2).
$$
Now let
\begin{align*}
  \gridlattice_1 & \;=\; \left\{ \bm{p_1} + q\bm{u} + r\bm{v} \;:\; 0\leqs q\leqs \tfrac{k+1}{2},\: 0\leqs r\leqs \tfrac{k-1}{2}  \right\}, \\[3pt]
  \gridlattice_2 & \;=\; \left\{ \bm{p_2} + q\bm{u} + r\bm{v} \;:\; 0\leqs q\leqs \tfrac{k-1}{2},\: 0\leqs r\leqs \tfrac{k+1}{2}  \right\}
\end{align*}
be two finite grids of lattice points.

We claim that $\gridlattice_1\cup\gridlattice_2$ is the plot of $\perm_k$.
Indeed,
for $1\leqs i\leqs m(k)$, if
$$
q \;=\; \floor{(i-1)/(k+2)} \quad \text{and} \quad r \;=\; (i-1)\!\!\!\mod(k+2),
$$
then
$$
  (i,\perm_k(i)) \;=\;
  \begin{cases}
  \bm{p_1} + q\bm{u} + r\bm{v}, & \text{ if $r\leqs\frac{k-1}{2}$, and}\\[2pt]
  \bm{p_2} + q\bm{u} + \left(r-\frac{k+1}{2}\right)\bm{v}, & \text{ otherwise.}
  \end{cases}
  $$

Similarly,
for even $k$, define the four vectors
$$
\bm{p_1} \:=\: (1,k+1)
, \quad
\bm{p_2} \:=\: \left(\tfrac{k}{2}+2,\tfrac{k}{2}\right)
, \quad
\bm{u} \:=\: (k+3,-1)
, \quad
\bm{v} \:=\: (1,k+1),
$$
and let
\begin{align*}
  \gridlattice_1 & \;=\; \left\{ \bm{p_1} + q\bm{u} + r\bm{v} \;:\; 0\leqs q\leqs \tfrac{k}{2},\: 0\leqs r\leqs \tfrac{k}{2}  \right\}, \\[3pt]
  \gridlattice_2 & \;=\; \left\{ \bm{p_2} + q\bm{u} + r\bm{v} \;:\; 0\leqs q\leqs \tfrac{k}{2}-1,\: 0\leqs r\leqs \tfrac{k}{2}+1  \right\} .
\end{align*}
Now, for $1\leqs i\leqs m(k)$, if
$$
q \;=\; \floor{(i-1)/(k+3)} \quad \text{and} \quad r \;=\; (i-1)\!\!\!\mod(k+3),
$$
then
$$
  (i,\perm_k(i)) \;=\;
  \begin{cases}
  \bm{p_1} + q\bm{u} + r\bm{v}, & \text{ if $r\leqs\frac{k}{2}$, and}\\[2pt]
  \bm{p_2} + q\bm{u} + \left(r-\frac{k}{2}-1\right)\bm{v}, & \text{ otherwise.}
  \end{cases}
  $$

  Thus, for each $k$, the plot of $\perm_k$ is given by $\gridlattice_1\cup\gridlattice_2$.
  We now use this characterization to bound $\minprol(k)$ from above.

\begin{thm}\label{thm:sigmak is kprolific}
For each positive integer $k$, the permutation $\perm_k$ is $k$-prolific.
\end{thm}

\begin{proof}
By Theorem~\ref{thm:kprolific iff large breadth}, we need only show that the breadth of $\perm_k$ is at least $k+2$.
Let $x$ and $y$ be distinct points of $\perm_k$.

Recall that the points in the plot of $\perm_k$ are partitioned into two sets, $\gridlattice_1$ and $\gridlattice_2$. If $x$ and~$y$ both lie in the same set, then their positions in the plot differ by a nonzero integer linear combination $q\bm{u} + r\bm{v}$. Thus, the $L_1$ distance between these points is given by
$$
\dwk(x,y) \;=\; \big|qa+r\big| \:+\: \big|rb-q\big|,
$$
where
$$
(a,b) \;=\; \begin{cases}
(k+2,k+2) & \text{ if $k$ is odd, and}\\
(k+3,k+1) & \text{ if $k$ is even.}
\end{cases}
$$

If $q = 0$, then $|r| \geqs 1$, because $x \neq y$. In that case, $\dwk(x,y) = |r|+|rb| \geqs 1 + b$.
Similarly, if $r=0$, then $|q| \geqs 1$ and $\dwk(x,y) = |qa|+|q| \geqs a + 1$.
Suppose now that both $|q|$ and $|r|$ are nonzero. Without loss of generality, we may assume that $r\geqs1$.
If $q\geqs1$, then $\dwk(x,y)\geqs|qa+r|\geqs a+1$.
If, on the other hand, $q\leqs-1$, then $\dwk(x,y)\geqs|rb-q|\geqs b+1$.
In each case, we have $\dwk(x,y)\geqs k+2$.

Now suppose, without loss of generality, that $x \in \gridlattice_1$ and $y \in \gridlattice_2$.
In this case, their positions in the plot differ by a vector $\bm{p_1} - \bm{p_2} + q \bm{u} + r \bm{v}$, for some integers $q$ and $r$.
Thus, the $L_1$ distance between these points is given by
$$
\dwk(x,y) \;=\; \big|qa-c+r\big| \:+\: \big|rb+d-q\big|,
$$
where
$$
(c,d) \;=\; \begin{cases}
\left(\tfrac{k+1}{2},\tfrac{k+3}{2}\right) & \text{ if $k$ is odd, and}\\
\left(\tfrac{k}{2}+1,\tfrac{k}{2}+1\right) & \text{ if $k$ is even.}
\end{cases}
$$

The possible distances can be
partitioned into the following five cases:
$$
\begin{array}{rcll}
  \dwk(x,y) & = & k + 2,                         & \quad \text{if $0\leqs q\leqs1$ and $-1\leqs r\leqs0$,} \\[3pt]
  \dwk(x,y) & \geqs & |qa-c+r| \;\geqs\; 2a-c-1, & \quad \text{if $q\geqs2$ and $r\geqs-1$,} \\[3pt]
  \dwk(x,y) & \geqs & |rb+d-q| \;\geqs\; b+d-1,  & \quad \text{if $q\leqs1$ and $r\geqs1$,}  \\[3pt]
  \dwk(x,y) & \geqs & |qa-c+r| \;\geqs\; a+c,    & \quad \text{if $q\leqs-1$ and $r\leqs0$,} \\[3pt]
  \dwk(x,y) & \geqs & |rb+d-q| \;\geqs\; 2b-d,   & \quad \text{if $q\geqs0$ and $r\leqs-2$.}
\end{array}
$$
Again, in each case, we have $\dwk(x,y)\geqs k+2$.

Therefore, the permutation $\perm_k$ has breadth $k+2$, and so, by Theorem~\ref{thm:kprolific iff large breadth}, it is $k$-prolific.
\end{proof}

This construction determines an upper bound
on the size of the smallest $k$-prolific permutation.
Together with the lower bound of
Theorem~\ref{thm:minprol lower bound}, it establishes the value of $\minprol(k)$ exactly.

\begin{cor}\label{cor:minprol}
  For each positive integer $k$,
  the smallest $k$-prolific permutations have size $\ceil{k^2/2+2k+1}$.
\end{cor}

\begin{cor}
  For each integer $d>2$,
  if $\DDD$ is a diamond whose diagonals have length $d$, then
  the smallest nontrivial permuted packings with tile $\DDD$ have size
  $\ceil{d^2/2-1}$.
\end{cor}

The first few terms of this sequence are $4, 7, 12, 17, 24, 31, 40, 49, 60, 71$.
It is sequence \href{https://oeis.org/A074148}{A074148} in~\cite{OEIS}.

Clearly, any symmetry of $\perm_k$ is $k$-prolific. However, for even $k\geqs6$, these permutations are not the only $k$-prolific permutations of minimal size.

See Figure~\ref{fig:6prolific perm} for an illustration of another $6$-prolific permutation of size $31$.
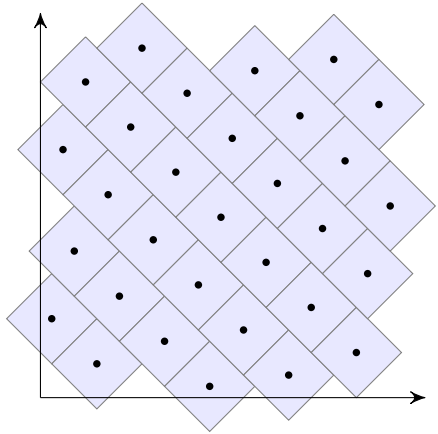
\begin{figure}[htbp]
\begin{tikzpicture}[scale=.15,>=stealthp] 
\foreach \x in
{(1,7),(2,22),(3,13),(4,28),(5,3),(6,18),(7,9),(8,24),(9,31),(10,14),(11,5),(12,20),(13,27),(14,10),(15,1),(16,16),(17,23),(18,6),(19,29),(20,12),(21,19),(22,2),(23,25),(24,8),(25,15),(26,30),(27,21),(28,4),(29,11),(30,26),(31,17)}
{
    \fill[blue!9] \x ++(0,4) -- ++(4,-4) -- ++(-4,-4) -- +(-4,4);
    \fill \x circle (0.33);
    \draw[gray] \x ++(0,4) -- ++(4,-4) -- ++(-4,-4) -- ++(-4,4) -- ++(4,4);
}
    \draw[->] (0,0) -- (34,0);
    \draw[->] (0,0) -- (0,34);
\end{tikzpicture}
\caption{An alternative 6-prolific permutation in $S_{31}$.}
\label{fig:6prolific perm}
\end{figure}

It is not immediately obvious that $k$-prolific permutations exist of every size greater than or equal to $\minprol(k)$.
We conclude this section by briefly presenting a construction that demonstrates that this is, in fact, the case.

\begin{thm}\label{thm:kprolific from minprol}
  There is a $k$-prolific permutation of every size greater than or equal to $\ceil{k^2/2+2k+1}$.
\end{thm}

\begin{proof}
We construct a $k$-prolific permutation, $\perm_k^{+j}$, of size $\minprol(k)+j$, for each $k\geqs1$ and $j\geqs0$.

Let $\perm_k^{+0}=\perm_k$. For $j\geqs0$, the permutation $\perm_k^{+j+1}$ is constructed by inserting a new first entry immediately above
the $(k+2)$th entry of $\perm_k^{+j}$, if $k$ is odd, or immediately above
the $(k+3)$th entry of $\perm_k^{+j}$, if $k$ is even.
See Figure~\ref{fig:growing perms} for an illustration.

We leave as an exercise for the reader
the rather tedious details of the proof that
this construction never leads to a reduction in the breadth of the permutation.

Furthermore, it can be shown that the breadth eventually increases:
if $k$ is odd, then
$\perm_k^{+k+2}$ is $(k+1)$-prolific, and, if $k$ is even, then
$\perm_k^{+k+3}$ is $(k+1)$-prolific.
\end{proof}

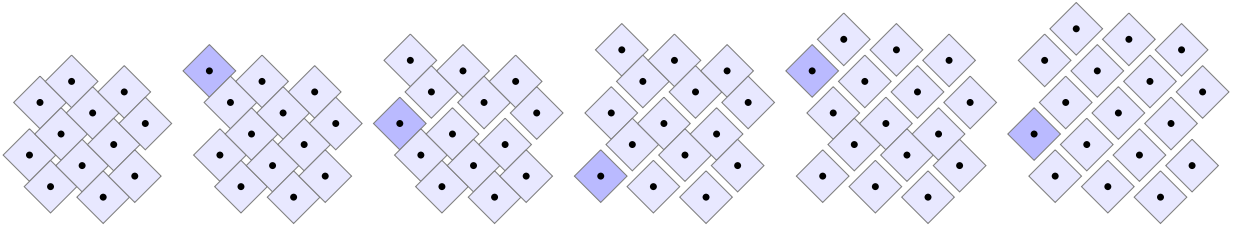
\begin{figure}[htbp]
\begin{tikzpicture}[scale=.14]
\foreach \x in
{(1,5),(2,10),(3,2),(4,7),(5,12),(6,4),(7,9),(8,1),(9,6),(10,11),(11,3),(12,8)}
{
    \fill[blue!9] \x ++(0,2.5) -- ++(2.5,-2.5) -- ++(-2.5,-2.5) -- +(-2.5,2.5);
    \fill \x circle (0.33);
    \draw[gray] \x ++(0,2.5) -- ++(2.5,-2.5) -- ++(-2.5,-2.5) -- ++(-2.5,2.5) -- ++(2.5,2.5);
}
\end{tikzpicture}
\begin{tikzpicture}[scale=.14]
\foreach \x in
{(1,13)}
{
    \fill[blue!27] \x ++(0,2.5) -- ++(2.5,-2.5) -- ++(-2.5,-2.5) -- +(-2.5,2.5);
    \fill \x circle (0.33);
    \draw[gray] \x ++(0,2.5) -- ++(2.5,-2.5) -- ++(-2.5,-2.5) -- ++(-2.5,2.5) -- ++(2.5,2.5);
}
\foreach \x in
{(2,5),(3,10),(4,2),(5,7),(6,12),(7,4),(8,9),(9,1),(10,6),(11,11),(12,3),(13,8)}
{
    \fill[blue!9] \x ++(0,2.5) -- ++(2.5,-2.5) -- ++(-2.5,-2.5) -- +(-2.5,2.5);
    \fill \x circle (0.33);
    \draw[gray] \x ++(0,2.5) -- ++(2.5,-2.5) -- ++(-2.5,-2.5) -- ++(-2.5,2.5) -- ++(2.5,2.5);
}
\end{tikzpicture}
\begin{tikzpicture}[scale=.14]
\foreach \x in
{(1,8)}
{
    \fill[blue!27] \x ++(0,2.5) -- ++(2.5,-2.5) -- ++(-2.5,-2.5) -- +(-2.5,2.5);
    \fill \x circle (0.33);
    \draw[gray] \x ++(0,2.5) -- ++(2.5,-2.5) -- ++(-2.5,-2.5) -- ++(-2.5,2.5) -- ++(2.5,2.5);
}
\foreach \x in
{(2,14),(3,5),(4,11),(5,2),(6,7),(7,13),(8,4),(9,10),(10,1),(11,6),(12,12),(13,3),(14,9)}
{
    \fill[blue!9] \x ++(0,2.5) -- ++(2.5,-2.5) -- ++(-2.5,-2.5) -- +(-2.5,2.5);
    \fill \x circle (0.33);
    \draw[gray] \x ++(0,2.5) -- ++(2.5,-2.5) -- ++(-2.5,-2.5) -- ++(-2.5,2.5) -- ++(2.5,2.5);
}
\end{tikzpicture}
\begin{tikzpicture}[scale=.14]
\foreach \x in
{(1,3)}
{
    \fill[blue!27] \x ++(0,2.5) -- ++(2.5,-2.5) -- ++(-2.5,-2.5) -- +(-2.5,2.5);
    \fill \x circle (0.33);
    \draw[gray] \x ++(0,2.5) -- ++(2.5,-2.5) -- ++(-2.5,-2.5) -- ++(-2.5,2.5) -- ++(2.5,2.5);
}
\foreach \x in
{(2,9),(3,15),(4,6),(5,12),(6,2),(7,8),(8,14),(9,5),(10,11),(11,1),(12,7),(13,13),(14,4),(15,10)}
{
    \fill[blue!9] \x ++(0,2.5) -- ++(2.5,-2.5) -- ++(-2.5,-2.5) -- +(-2.5,2.5);
    \fill \x circle (0.33);
    \draw[gray] \x ++(0,2.5) -- ++(2.5,-2.5) -- ++(-2.5,-2.5) -- ++(-2.5,2.5) -- ++(2.5,2.5);
}
\end{tikzpicture}
\begin{tikzpicture}[scale=.14]
\foreach \x in
{(1,13)}
{
    \fill[blue!27] \x ++(0,2.5) -- ++(2.5,-2.5) -- ++(-2.5,-2.5) -- +(-2.5,2.5);
    \fill \x circle (0.33);
    \draw[gray] \x ++(0,2.5) -- ++(2.5,-2.5) -- ++(-2.5,-2.5) -- ++(-2.5,2.5) -- ++(2.5,2.5);
}
\foreach \x in
{(2,3),(3,9),(4,16),(5,6),(6,12),(7,2),(8,8),(9,15),(10,5),(11,11),(12,1),(13,7),(14,14),(15,4),(16,10)}
{
    \fill[blue!9] \x ++(0,2.5) -- ++(2.5,-2.5) -- ++(-2.5,-2.5) -- +(-2.5,2.5);
    \fill \x circle (0.33);
    \draw[gray] \x ++(0,2.5) -- ++(2.5,-2.5) -- ++(-2.5,-2.5) -- ++(-2.5,2.5) -- ++(2.5,2.5);
}
\end{tikzpicture}
\begin{tikzpicture}[scale=.14]
\foreach \x in
{(1,7)}
{
    \fill[blue!27] \x ++(0,2.5) -- ++(2.5,-2.5) -- ++(-2.5,-2.5) -- +(-2.5,2.5);
    \fill \x circle (0.33);
    \draw[gray] \x ++(0,2.5) -- ++(2.5,-2.5) -- ++(-2.5,-2.5) -- ++(-2.5,2.5) -- ++(2.5,2.5);
}
\foreach \x in
{(2,14),(3,3),(4,10),(5,17),(6,6),(7,13),(8,2),(9,9),(10,16),(11,5),(12,12),(13,1),(14,8),(15,15),(16,4),(17,11)}
{
    \fill[blue!9] \x ++(0,2.5) -- ++(2.5,-2.5) -- ++(-2.5,-2.5) -- +(-2.5,2.5);
    \fill \x circle (0.33);
    \draw[gray] \x ++(0,2.5) -- ++(2.5,-2.5) -- ++(-2.5,-2.5) -- ++(-2.5,2.5) -- ++(2.5,2.5);
}
\end{tikzpicture}
\caption{Plots of the 3-prolific permutations $\perm_3^{+j}$, for $j=0,\ldots,5$; note that $\perm_3^{+5}$ is, in fact, 4-prolific.}
\label{fig:growing perms}
\end{figure}

\section{Directions for further research}

In Section~\ref{section:constructions}, we noted that $\perm_k$ and its symmetries were not necessarily the only $k$-prolific permutations of minimal size.
However, for odd $k$, no additional $k$-prolific permutations of size $\minprol(k)$ are known.
This prompts the following conjecture.

\begin{conj}
  For each odd $k$, the permutation $\perm_k$ (described in Definition~\ref{defn:a short prolific permutation}) and its symmetries are the only $k$-prolific permutations of minimal size.
\end{conj}

More generally, we wonder whether it is possible to enumerate and characterize all minimal $k$-prolific permutations.

\begin{question}
  For each $k$, how many distinct $k$-prolific permutations of minimal size are there, and what are they?
\end{question}

Another topic of potential interest concerns the presence of $k$-prolific permutations in specific permutation classes (sets closed downwards in the pattern poset $\PPP$).
For example, there appear to be no $1$-prolific permutations avoiding $132$, and no $2$-prolific permutations avoiding $123$.
This motivates the following question.
\begin{question}
  For each $k$, which principal permutation classes (those avoiding a single pattern) contain $k$-prolific permutations?
\end{question}

In various guises, the enumeration of $1$-prolific permutations (sequence \href{https://oeis.org/A002464}{A002464} in~\cite{OEIS}) has been well-studied ever since Kaplansky's 1944 paper addressing the ``$n$ king problem''~\cite{Kaplansky1944,Kaplansky1945}. Tauraso~\cite{Tauraso2006} presents complete asymptotics.
For large $n$, the proportion of permutations of size $n$ which are $1$-prolific  is
$$
e^{-2} \left( 1 \:-\: \frac{2}{n^2} \:-\: \frac{10}{3n^3} \:-\: \frac{6}{n^4} \:-\: \frac{154}{15n^5} \:+\: O\left(\frac{1}{n^6}\right)  \right) .
$$

However, nothing specific appears to have been published concerning the enumeration of $k$-prolific permutations for larger $k$. 
In a forthcoming paper, Blackburn, Homberger and Winkler establish that the
proportion of permutations of size $n$ which are $k$-prolific is asymptotically $e^{-k^2-k}$ (see~\cite{BHW2017+}).
\begin{question}
  For a given $k>1$, how does the number of $k$-prolific permutations of size $n$ grow with $n$?
\end{question}

The notion of being $k$-prolific can also be transferred to the context of other graded posets, an element of rank $n$ being $k$-prolific if
it has maximally many children of rank $n-k$.
The characterization of the $k$-prolific elements of various combinatorial posets,
perhaps most obviously those relating to the various subgraph orders, may be of interest.

Finally, permuted packings
also invite
further investigation.
In addition to the permuted diamond packings studied here, one might consider permuted packings of other regular tiles.
Permuted packings of axis-parallel squares
appear uninteresting.
On the other hand, permuted circle packings raise some intriguing questions.
See Figure~\ref{fig:permuted circles} for an illustration.
\begin{figure}[htbp]
\raisebox{21pt}
{
\begin{tikzpicture}[scale=.2]
\foreach \x in
{(1,4),(2,8),(3,12),(4,1),(5,5),(6,9),(7,13),(8,2),(9,6),(10,10),(11,14),(12,3),(13,7),(14,11)}
{
    \fill[blue!9] \x circle (2.06155);
}
\draw[help lines] (1,1) grid (14,14);
\foreach \x in
{(1,4),(2,8),(3,12),(4,1),(5,5),(6,9),(7,13),(8,2),(9,6),(10,10),(11,14),(12,3),(13,7),(14,11)}
{
    \fill \x circle (0.275);
    \draw[gray] \x circle (2.06155);
}
\end{tikzpicture}
} 
$\qquad\qquad$
\begin{tikzpicture}[scale=.2]
\foreach \x in
{(1,10),(2,15),(3,4),(4,19),(5,8),(6,13),(7,2),(8,17),(9,6),(10,21),(11,11),(12,1),(13,16),(14,5),(15,20),(16,9),(17,14),(18,3),(19,18),(20,7),(21,12)}
{
    \fill[blue!9] \x circle (2.12133);
}
\draw[help lines] (1,1) grid (21,21);
\foreach \x in
{(1,10),(2,15),(3,4),(4,19),(5,8),(6,13),(7,2),(8,17),(9,6),(10,21),(11,11),(12,1),(13,16),(14,5),(15,20),(16,9),(17,14),(18,3),(19,18),(20,7),(21,12)}
{
    \fill \x circle (0.275);
    \draw[gray] \x circle (2.12133);
}
\end{tikzpicture}
\caption{Minimal permuted packings of circles of diameter $\sqrt{17}$ and $\sqrt{18}+\veps$.}
\label{fig:permuted circles}
\end{figure}
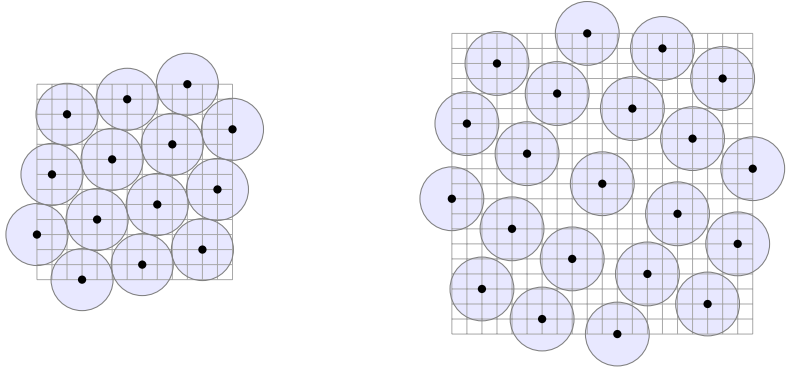

Recall that
the \emph{density} of a packing $\Pi$ relative to a bounded domain $D$ is defined as
$$
d(\Pi,D) \;=\; \frac{\sum_{\TTT\in\Pi}\mu(\TTT\cap D)}{\mu(D)},
$$
where $\mu(X)$ is the area of $X$
(see~\cite{brass:discretegeometry}).

Let us call a permuted packing of minimal cardinality a \emph{minimal} permuted packing.
Among other problems, one that is particularly attractive would be to determine how poor a minimal permuted packing can be, asymptotically as the radius of the circular tiles tends to infinity.
\begin{question}
  What is the value of
  $$
  \liminf_{\rho\to\infty} d(\Pi_\rho,[1,n_\rho]^2),
  $$
  where $\Pi_\rho$ is a minimal permuted packing of circles of radius $\rho$, and $n_\rho$ is the number of circles in such a packing?
\end{question}

Similar questions might be asked about permuted packings of regular hexagons.

\section*{Acknowledgements}
The authors are grateful for the feedback from the referees, which led to improvements in the presentation.
The first author would like to thank Judy and Iska Routamaa for their hospitality and encouragement
when he was carrying out the research for this paper.

\bibliography{prolific}
\bibliographystyle{acm}

\end{document}